\documentclass[11pt]{amsart}
\usepackage{amsmath,amssymb,amsfonts,amsthm}
\usepackage{fullpage}
\usepackage{pdfsync}
 \usepackage{color}
\newtheorem{theorem}{Theorem}
\newtheorem{corollary}{Corollary}
\newtheorem{lemma}{Lemma}
\newtheorem{proposition}{Proposition}
\newtheorem{remark}{Remark}
\makeatletter

\@addtoreset{equation}{section}
\makeatother

\newcommand{\N}{{\mathbb N}}

\newcommand{\R}{{\mathbb R}}

\newcommand{\pa}{{\partial}}
\newcommand{\na}{{\nabla}}
\newcommand{\dis}{{\displaystyle}}
\newcommand{\eps}{{\varepsilon}}

\newcommand{\I}{\mathcal{I}}
\renewcommand{\a}{\alpha}

\newcommand{\D}{\Delta}

\newcommand{\om}{\omega}

\newcommand{\e}{\varepsilon}

\def\curl{\hbox{curl \!}}
\def\div{\hbox{div  }}

\def\np{\dot{n}}
\def\up{\dot{u}}
\def\fp{\dot{\phi}}

\def\ZZ{\mathcal{Z}}

\newcommand{\Black}{\color{black}}

\begin{document}

\title{Quasineutral limit of the Euler-Poisson system for ions\\ in a domain with boundaries II}
\author{David G\'erard-Varet}
\address{D. G\'erard-Varet \\ Institut de Math\'ematiques de Jussieu (UMR 7586), Universit\'e Paris-Diderot} 
\email{gerard-varet@math.jussieu.fr}
\author{Daniel Han-Kwan}
 \address{D. Han-Kwan \\ CNRS $\&$ Centre de Math\'ematiques Laurent Schwartz (UMR 7640), \'Ecole Polytechnique}
\email{daniel.han-kwan@math.polytechnique.fr}
 \author{Fr\'ed\'eric Rousset}
  \address{F. Rousset \\ Laboratoire de Math\'ematiques d'Orsay (UMR 8628), Universit\'e Paris-Sud et Institut Universitaire de France}
  \email{frederic.rousset@math.u-psud.fr}
  
  \maketitle
  
  \begin{abstract}
In this paper, we study the quasineutral limit of the isothermal Euler-Poisson equation for ions, in a domain with boundary. 
This is a follow-up to our previous work \cite{GVHKR}, devoted to  no-penetration  as well as subsonic outflow boundary conditions. We focus here on the case of supersonic outflow velocities. The structure of the boundary layers and the stabilization mechanism are different.

\end{abstract}
  
 \section{Introduction}
  
  This work is about the quasineutral limit of the isothermal Euler-Poisson equation,  describing ions in a plasma. We focus  on the role of  the boundary of  the plasma domain, and the associated boundary layer.  It complements our previous work on the topic \cite{GVHKR} (see also \cite{CG}, \cite{SS}).  We refer to the introduction of \cite{GVHKR} for a substantial  bibliography on the quasineutral limit and related issues. 

In the Euler-Poisson model, ions are described by their density $n \geq 0$ and their velocity field $u \in \R^3$, while electrons are assumed to follow a \emph{Maxwell-Boltzmann} law, i.e. their density is given by $e^{-\phi}$, where $\phi$ denotes the electric potential. We consider that the plasma is contained in the domain $\R^3_+ := \{ x=(y,z) \in \R^2 \times \R^+\}$ (yet, general domains can actually be considered, see \cite[Section 5.1]{GVHKR}). The isothermal Euler-Poisson equation under study then reads, for $(t,x) \in \R^+ \times \R^3_+$,
  \begin{equation} \label{EP}
\left\{
\begin{aligned}
& \pa_t n \: + \:  \div(n u) \:  = \: 0, \\
& \pa_t u  \: + \:  u \cdot \na u  \: + \: T^i \,  \na \ln(n)= \na \phi, \\
& \eps^2 \Delta \phi \: + \: e^{-\phi}  = n, \\
& \phi\vert_{x_3 = 0} = \phi_b,
\end{aligned}
\right. 
\end{equation}
where $\eps>0$ is loosely speaking the ratio between the \emph{Debye length} of the plasma and the typical length of observation. In all practical situations, it is a small parameter. The parameter $T^i>0$ is the temperature of the ions. We consider in this work the case of supersonic outflow velocities. It means that we consider initial velocity fields $u(0)=(u_1(0),u_2(0),u_3(0))$ satisfying 
  $$
  u_3(0,y,0) < -\sqrt{T^i}
  $$
  in which case no boundary condition is needed for the Euler system, since there are only outgoing characteristics. On the other hand, we enforce a Dirichlet boundary condition on the electric potential. 

 We also fix some constant reference state: $n_{ref} > 0$, $u_{ref} = (0,0,w_{ref})$ with $w_{ref} < 0$. 
 We set $\phi_{b}= \phi_{c} + (\phi_{ref})_{|x_{3}= 0}$, with $\phi_{ref} = -\ln(n_{ref})$, and $\phi_c \in H^\infty(\R^2)$.     
  
  We are interested in the behaviour of the solutions of \eqref{EP} as $\eps$ goes to $0$, that is in the so-called \emph{quasineutral limit}. We refer once again to the introduction of \cite{GVHKR} for an extensive discussion. The expected formal limit, obtained by taking directly $\eps=0$, is the isothermal Euler system:
    \begin{equation} \label{EI}
\left\{
\begin{aligned}
& \pa_t n \: + \:  \div(n u) \:  = \: 0, \\
& \pa_t u  \: + \:  u \cdot \na u  \: + \: (T^i+1) \,  \na \ln(n)= 0,
\end{aligned}
\right. 
\end{equation}
  together with the neutrality relation $n = e^{-\phi}$. With regards to this last relation and  the Dirichlet condition (\ref{EP}d) on $\phi$, one would like to impose the boundary condition   $n\vert_{x_3 = 0} = e^{-\phi_b}$. However, this condition is not {\it a priori} compatible with the hyperbolic system \eqref{EI}, therefore we expect the formation of a \emph{boundary layer} in the vicinity of the boundary $\{x_3=0\}$. \Black

  In our previous work \cite{GVHKR}, we have considered subsonic outflow velocities. In this setting, a  boundary condition on the normal velocity  has to be added:
  $$
  u \cdot n = \overline{u}.
  $$
  In \cite{GVHKR}, we have notably focused on the non-penetration boundary condition (that is $\overline{u}=0$), and we have also considered the subsonic case $-\sqrt{T^i} < \overline{u}<0$.

In this paper, we shall concentrate on the supersonic case
\begin{equation}
\label{bohmintro}
  u_3(0,y,0) < -\sqrt{T^i+1},
\end{equation}
 We will comment briefly on  the intermediate case  $-\sqrt{T^i+1}<u_3(0,y,0) < -\sqrt{T^i}$, see  Section~\ref{noBL}.
  
  The supersonic condition \eqref{bohmintro} is usually called  \emph{Bohm condition} (or \emph{criterion}), while the boundary layer is usually referred to as the \emph{sheath} in the physics literature. It plays an important role in the study of confined plasmas: see for instance the book of Lieberman and Lichtenberg \cite[Chapter 6]{LL} and  the review paper of Riemann \cite{Rie}.

  \subsection{Main result}

  The main result proved in this paper is the following Theorem.
  \begin{theorem}
\label{theomain}
Let $(n^{0}, u^{0})$  a solution to \eqref{EI} such that
$(n^0- n_{ref}, u^0- u_{ref})\in C^0([0,T],H^s(\R^3_{+}))$ 
 with $s$ large enough. We assume that 
 $$ \sup_{[0, T]}  \sup_{y \in \mathbb{R}^2}\big(u^0_3(t,y, 0 )  + \sqrt{ T^i+ 1 } \big) < 0$$
  and that 
  \begin{equation}
  \label{smallintro} \sup_{[0, T]}  \sup_{y \in \mathbb{R}^2} \big( |n^0(t,y,0) - e^{-\phi_{b}} | +  |u_{1, 2}(t,y,0)| \big) \leq \delta, 
  \end{equation}
  for some sufficiently small $\delta$.
   Then, there is $\e_0>0$ such that for any $\e \in (0,\e_0]$, there exists $(n^\e,u^\e)$ a solution to \eqref{EP} also defined on $[0,T]$ such that
 \begin{equation*}
  \begin{aligned}
&\sup_{[0,T]}\left( \| n^\e- n^0 \|_{L^2(\R^3_+)}  +   \| u^\e- u^0 \|_{L^2(\R^3_+)}\right) \rightarrow_{\eps \rightarrow 0} 0.
  \end{aligned}
  \end{equation*}
  Furthermore, the rate of convergence is $O(\sqrt{\e})$.
    \end{theorem}
Note that the supersonic condition forbids $y \mapsto u^0_3(t,y,0)$ to vanish at infinity :  more precisely, it behaves asymptotically in $y$ like $w_{ref}$, which necessarily satisfies  $w_{ref} < -  \sqrt{ T^i+ 1 }$. 

The smallness condition \eqref{smallintro} is used in particular to  ensure the stability of a boundary layer. The study of this boundary layer, which we build beforehand in Theorem~\ref{deriv}, is crucial in the proof of Theorem \ref{theomain}, though it does not appear explicitly. We refer to Theorem \ref{main} and Corollary \ref{coro} for more precise statements.
 
 To prove Theorem~\ref{theomain} (and its refined version), we shall use a classical two-step argument  as in our previous work \cite{GVHKR}. The first step is a \emph{consistency} step where we first build approximate solutions for \eqref{EP} at a sufficiently high order, see Theorem~\ref{deriv}. In the second step we combine linear estimates and a continuation argument to deduce the \emph{stability} of those approximate solutions. The main differences compared to the analysis of \cite{GVHKR} are as follows.
 \begin{enumerate}
 \item Contrary to \cite{GVHKR}, the existence of the sequence of approximate solutions is not unconditional, and relies on the first part of the smallness assumption \eqref{smallintro}. Furthermore, at the main order, compared to the subsonic case treated in \cite{GVHKR}, there is a boundary layer for the third (i.e. the normal) component of the velocity.
 
 \item Because of this additional singular term, the linear estimates of \cite{GVHKR} are not relevant. Instead, we shall consider weighted estimates, inspired by the classical work of  Goodman  \cite{Goo}.  The idea is  to use the stabilizing effect of convection. We shall also borrow some ideas  from a recent paper of Nishibata, Ohnawa and Suzuki \cite{NOS} on a related problem. Namely, this problem is  the stability of a special solution of the unscaled Euler-Poisson system, that is when $\eps = 1$.  Roughly, this special solution corresponds to the main order part of the boundary layer constructed in Theorem~\ref{deriv}, but without any $y$- or $\eps$-dependence.  On that topic,  one can also refer to the previous work of Suzuki \cite{Suz}, as well as the works of  Ambroso, M\'ehats and Raviart \cite{AMR}, and Ambroso \cite{AMB}).
 
Our analysis will combine three types of  $L^2$ weighted estimates. In \cite{GVHKR}, only the ``physical'' unweighted  energy estimate was needed. This was due to the fact
 that the physical energy  which is well adapted to symmetrize the  singular term coming from the Poisson equation in the quasineutral limit was compatible with
  the structure of the boundary layers. Here this is not the case and we have to use the  the stabilizing effect of convection via the weights and other energy functionals   in order to control the new
   singular terms that arise.
 \end{enumerate}

\subsection{About the case $-\sqrt{T^i+1}<u_3(0,y,0) < -\sqrt{T^i}$}
\label{noBL}
We shall explain in this very brief section how the intermediate case $-\sqrt{T^i+1}<u_3(0,y,0) < -\sqrt{T^i}$ can be handled. In this regime, no boundary condition can be imposed when $\eps > 0$, but one boundary condition can (and actually must) be imposed when $\eps = 0$, that is for system \eqref{EI}. We can notably endow \eqref{EI} with the boundary condition   $n\vert_{x_3 = 0} = e^{-\phi_b}$. This additional boundary condition allows to satisfy the Dirichlet condition $\phi\vert_{x_3 = 0} = \phi_b$ in the limit $\eps \rightarrow 0$, hence no boundary layer appears. Convergence of solutions of \eqref{EP} to the solution of \eqref{EI} is in this context straightforward to prove: one can build an accurate approximate solution of \eqref{EP}, whose first term is the solution of \eqref{EI}. As explained before, this approximate solution will not have any boundary layer part. Then, one can perform an energy estimate between the approximate and exact solutions of \eqref{EP} (with the same initial data) to show convergence. In this way, we get

  \begin{proposition}
\label{theopasmain}
Let $(n^{0}, u^{0})$  a solution to \eqref{EI} with the condition $n\vert_{x_3 = 0} = e^{-\phi_b}$, such that
$(n^0- n_{ref}, u^0-u_{ref})\in C^0([0,T],H^s(\R^3_{+}))$ 
 with $s$ large enough. We assume that 
 $$ \sup_{[0, T]}  \sup_{y \in \mathbb{R}^2}\big(u^0(t,y, 0 )_{3}  + \sqrt{ T^i } \big) < 0, \quad \inf_{[0, T]}  \inf_{y \in \mathbb{R}^2}\big(u^0(t,y, 0 )_{3}  + \sqrt{ T^i+ 1 } \big) > 0.$$
   Then, there is $\e_0>0$ such that for any $\e \in (0,\e_0]$, there exists $(n^\e,u^\e)$ a solution to \eqref{EP} also defined on $[0,T]$ such that
 \begin{equation*}
  \begin{aligned}
&\sup_{[0,T]}\left( \| n^\e- n^0 \|_{L^2(\R^3_+)}  +   \| u^\e- u^0 \|_{L^2(\R^3_+)}\right) \rightarrow_{\eps \rightarrow 0} 0, \\
 & \sup_{[0,T]}\left( \| n^\e- n^0 \|_{L^\infty(\R^3_+)}  +   \| u^\e- u^0 \|_{L^\infty(\R^3_+)}\right) \rightarrow_{\eps \rightarrow 0} 0. \Black
  \end{aligned}
  \end{equation*}
  Furthermore, the rate of convergence is $O({\e})$. 
    \end{proposition}
Note that we get convergence in $L^\infty$ due to the absence of boundary layers. In the setting of Theorem \ref{theomain}, boundary layers have to be included
 in order to describe the asymptotic behavior in $L^\infty$ (see the refined version  of the convergence in Corollary  \ref{coro}).

\subsection{Overview and classification}
Summarizing the results obtained in \cite{GVHKR} and in the present paper, we obtain the following classification of outflow boundary conditions for the study of the quasineutral limit of the Euler-Poisson equation.

\begin{center}

\begin{tabular}{|c|c|c|c|}
   \hline
   Boundary condition & Boundary condition &  Main order & Stability result\\
    for Euler-Poisson \eqref{EP} & for isothermal Euler \eqref{EI} & boundary layer & \\  ($\eps>0$) & ($\eps=0$)   & &\\
   \hline
$u\cdot n =0$ $\&$ $\phi=\phi_b$ & $u\cdot n =0$ & Density & \cite[Theorem 2.1]{GVHKR} \\(characteristic) &  & and Potential &   \\
\hline
 $u\cdot n =\overline{u}$, $\overline{u}<0$ $\&$ $\phi=\phi_b$  &  & &  \\ $|\overline{u}| < \sqrt{T^i}$  & $u\cdot n =\overline{u}$ & Density  & \cite[Section 5.2]{GVHKR}  \\(non-characteristic)  &  & and Potential &  \\
 \hline
$\phi=\phi_b$  &   & &    \\ and initial datum s.t. &  $n = e^{-\phi_b}$ & No boundary layer  & Proposition \ref{theopasmain} \\  $-\sqrt{T^i+1}<u_3(0,y,0) < -\sqrt{T^i}$   & & & \\(non-characteristic)     &  & &  \\
 \hline
$\phi=\phi_b$  &  &  & \\ and initial datum s.t. & None & Density, Potential  &  Theorem  \ref{theomain}    \\ $u_3(0,y,0) < -\sqrt{T^i+1}$  &   & and Velocity   &  (under smallness\\(non-characteristic)        & &    & assumption \eqref{smallintro})\\
\hline
\end{tabular}

\end{center}

\bigskip

The rest of the paper is now entirely devoted to the proof of Theorem~\ref{theomain} (more precisely  of the refined Theorem \ref{main}).

  \section{Derivation of the boundary layers}
  We construct in this section accurate approximate solutions of the Euler Poisson system, of boundary layer type. They are expansions in powers of $\eps$, of the form: 
 \begin{equation} \label{ansatz}
\begin{aligned}
(n_{a},u_{a}, \phi_{a})    & \: = \:  \sum_{i=0}^K \eps^i \left( n^i(t,x), u^i(t,x), \phi^i(t,x) \right) \\  & \: + \: \sum_{i=0}^K \eps^i \left( N^i\left(t,y,\frac{x_3}{\eps}\right), U^i\left(t,y,\frac{x_3}{\eps}\right), \Phi^i\left(t,y,\frac{x_3}{\eps}\right)\right), 
\end{aligned}
\end{equation}
with $K$ an arbitrarily large integer. These approximations split into two parts:
\begin{itemize}
\item a regular part, with  coefficients $(n^i, u^i, \phi^i)$  depending on $(t,x)$. This regular part models the macroscopic behaviour of the solutions.
\item a singular part, with coefficients $(N^i, U^i, \Phi^i)$ depending on the regular variables $t,y$, but also on a rescaled variable $z = \frac{x_3}{\eps} \in \R_+$. It models a boundary layer correction, of  size $\eps$ near the boundary. Accordingly, we shall impose
\begin{equation} \label{decay}
 (N^i, U^i, \Phi^i) \rightarrow 0, \quad \mbox{  as }  \: z \rightarrow +\infty.
 \end{equation}
 \end{itemize}
 In order for  the whole approximation $\phi_a$ to satisfy the proper Dirichlet condition, we shall  further impose
\begin{equation} \label{dirichlet}
 \phi^0(t,y,0) \: + \: \Phi^0(t,y,0) \: = \:  \phi_b, \quad   \phi^i(t,y,0) \: + \: \Phi^i(t,y,0) \: = \:  0   \quad \mbox{for all } i \ge 1.       
\end{equation}
Note that, far away from the boundary,  the term $(n^0, u^0, \phi^0)$  will drive the dynamics of these approximate solutions, and will satisfy the quasineutral limit system: 
\begin{equation} \label{E}
\left\{
\begin{aligned}
 & \pa_t n \: + \:  \div(n u) \:  = \:  0, \\
 & \pa_t u  \: + \:  u \cdot \na u  \: + \: (T^i+1) \na \ln(n)= 0,  
\end{aligned}
\right. 
\end{equation}
together with the relation $n^0 = e^{-\phi^0}$.  
\bigskip

Our construction is summarized in the
\begin{theorem}
\label{deriv}
 There exists $\delta_0 > 0$ such that: for any $K \in \N^*$, for any $\dis (n^0_0,u^0_0)$ such that $ \inf n^0_0 > 0$  satisfying
 the regularity assumption
\begin{equation} \label{initialregularity}
(n^0_{0}- n_{ref}, u^0_{0} - u_{ref}) \in H^\infty(\R^3_+), 
\end{equation} 
the Bohm condition
\begin{equation} \label{supersonic}
 \sup u^0_{0,3}  < 0, \quad (\sup u^0_{0,3}))^2  >T^i + 1,
\end{equation}
and the smallness condition
   \begin{equation} \label{smallness}
      \phi_c \: +  \: \inf_{y \in \R^2}   \ln\left(\frac{n^0_0(y,0)}{n_{ref}}\right)   \: \ge \: - \delta_0, \quad    \phi_c \: + \: \sup_{y \in \R^2}   \ln\left(\frac{n^0_0(y,0)}{n_{ref}}\right)    \: \le \:  \delta_0
\end{equation}
one can find $T>0$ and an expansion of type \eqref{ansatz}  with the following properties: 

\medskip
\noindent
{\bf i)} $(n^0,u^0)$ is a  solution to  \eqref{E} with initial data $(n^0_0,u^0_0) $, satisfying 
$$ (n^0- n_{ref}, u^0 - u_{ref}) \in C^\infty([0,T],H^\infty(\R^3_+)). $$ 
Moreover,  $\phi^0=- \log n^0$.  

\medskip
\noindent
{\bf ii)}  $\forall \, 1\leq i \leq K$, $(n^i, u^i, \phi^i)  \in  C^\infty([0,T], H^{\infty}(\R^3_+)).$

\medskip
\noindent
{\bf iii)} $\forall \, 0\leq i \leq K, \, (N^i,U^i,\Phi^i) \in    C^\infty([0,T], H^\infty_{y,z}(\R^3_+))$  and decays together with its  derivatives  uniformly exponentially in $z$.
  
\medskip
\noindent
{\bf iv)} Let us consider $(n^\eps, u^\eps, \phi^\eps)$ a solution to \eqref{EP} and define:
\begin{equation*}
n  \: = \:  n^\eps \:  - n_a , \quad u  \: = \: u^\eps  \: -  u_a , \quad \phi  \: = \:  \phi^\eps \: - \phi_a.
\end{equation*}
Then $(n,u,\phi)$ satisfies the system of equations:
\begin{equation} \label{approxEP}
\left\{
\begin{aligned}
& \pa_t  n +  (u_a + u ) \cdot \nabla n + n \,\div (u + u_{a})    + \div  (n_{a} u)  = \eps^{K} R_n,  \\
& \pa_t u + (u_a + u)  \cdot \na u + u \cdot \nabla u_{a}  + T^i   \left(\frac{\na n}{n_a +  n } - \frac{\na n_a}{n_a} \left( \frac{n}{n_a +  n} \right) \right) =  \na \phi + \eps^{K}R_u, \\
& \eps^2 \Delta \phi = n - e^{-\phi_a}( e^{ - \phi} - 1 \big) + \eps^{K+1} R_\phi.
\end{aligned}
\right.
\end{equation}
where $R_n,R_u,R_\phi$ are remainders satisfying:
\begin{equation}
\label{estirestes}
\sup_{[0,T]}\| \nabla_x^\a R_{n,u,\phi} \|_{L^2(\R^3_+)} \leq C_\a \eps^{-\a_3}, \quad \forall \a=(\a_1,\a_2,\a_3)\in\N^3, \quad |\a|\leq m,
\end{equation}
with $C_\a>0$ independent of $\e$.
\end{theorem}
  
The whole section is devoted to the proof of this theorem. As already said in the introduction, the main difference with the boundary layer problem analyzed in \cite{GVHKR} is that the third component of $U^0$ does not vanish: the velocity field has a boundary layer part of amplitude $O(1)$. This modifies all the boundary layer equations, compared to \cite{GVHKR}. In particular, well-posedness is not unconditional anymore, which explains our condition \eqref{smallness}. In fact, the construction of the first boundary layer term only requires a lower bound (first inequality in \eqref{smallness}), whereas the construction of the next terms requires an additional upper bound. Moreover,   the constant $\delta_0$ in this condition can be made explicit:   see  paragraph \ref{subsecwp}.  
Let us finally point out that the regularity assumption on the initial data $(n^0_0, u^0_0)$ can be easily lowered to $H^s$, with $s$ large enough.

\subsection{The cascade of equations}
Plugging \eqref{ansatz} in the Euler-Poisson  system \eqref{EP}, we formally derive a whole set of equations. The well-posedness of  such equations will be discussed in the next paragraph. For clarity of exposure, for any function $f=f(t,x)$, we denote by    
 $\Gamma f$ the function  $(t,y,z) \mapsto  f(t,y,0)$.

\medskip
a) Away from the boundary, we find as expected that $(n^0,u^0)$ satisfies \eqref{E}, plus  the neutrality relation $\phi^0 = - \ln (n^0)$. The local existence and uniqueness of $(n^0,u^0)$, starting from  our supersonic data, will be stated rigorously in the next paragraph.
 
\medskip
b) In the boundary layer, with a Taylor expansion of the regular part of \eqref{ansatz}, the third line of \eqref{EP} yields
\begin{equation} \label{Phi0}
\pa^2_z \Phi^0 = (\Gamma n^0 + N^0) - e^{-(\Gamma \phi^0 + \Phi^0)} = \Gamma n^0 \left( \frac{\Gamma n^0 + N^0}{\Gamma n^0} - e^{-\Phi^0} \right),
\end{equation}
whereas the first line yields
$$ \pa_z \left( (\Gamma n^0 + N^0) (\Gamma u^0_3 + U^0_3 ) \right) = 0.$$
Taking into account \eqref{decay}, we obtain
\begin{equation} \label{N0U03}
 (\Gamma n^0 + N^0) (\Gamma u^0_3 + U^0_3 ) \: = \:  \Gamma n^0 \, \Gamma u^0_3. 
 \end{equation}
 We then consider the vertical component of the momentum equation in \eqref{EP}. This gives
 $$ (\Gamma u^0_3 + U^0_3 ) \pa_z   U^0_3  + T^i \pa_z \ln(\Gamma n^0 + N^0)  \:= \: \pa_z \Phi^0. $$
 We can integrate in $z$, and use \eqref{decay} again to obtain
 $$ \frac{1}{2}  (\Gamma u^0_3 + U^0_3 )^2 + T^i  \ln(\Gamma n^0 + N^0) = \Phi^0 +   \frac{1}{2}  (\Gamma u^0_3)^2 + T^i  \ln(\Gamma n^0). $$
 Inserting relation \eqref{N0U03} in the last identity,  we end up with 
 \begin{equation}
  \frac{(\Gamma n^0 \, \Gamma u^0_3)^2}{2 (\Gamma n^0 + N^0)^2} \:  + \: T^i   \ln(\Gamma n^0 + N^0)  \: = \:  \Phi^0  \:+ \:   \frac{1}{2}  (\Gamma u^0_3)^2 + + T^i  \ln(\Gamma n^0).
  \end{equation}
  This last equation can be written 
\begin{equation}  \label{Phi0N0}
  \Phi^0(t,y,z) = \mathcal{F}\left(t,y, \frac{ \Gamma n^0 (t,y)  + N^0(t,y,z)}{\Gamma n^0 (t,y)} \right)
 \end{equation}
 where 
\begin{equation} \label{defF} 
 \mathcal{F}(t,y,N) =  \frac{\Gamma u^0_3(t,y)^2}{2 N^2} \: + \: T^i \ln(N) \: - \:   \frac{\Gamma u^0_3(t,y)^2}{2}. 
\end{equation}
The derivative with respect to $N$ is 
$$\pa_N  \mathcal{F}(t,y,N) = -  \frac{\Gamma u^0_3(t,y)^2}{ N^3} \: + \: \frac{T^i}{N}.  $$
One has
$$ \pa_N  \mathcal{F}(t,y,N) < 0 \quad \mbox{over }  \: ]0, N_ \mathcal{F}(t,y)[, \quad  N_ \mathcal{F}(t,y) :=   \sqrt{\Gamma u^0_3(t,y)^2/T^i}.   $$
Hence, for all $t,y$,   $\: N \mapsto   \mathcal{F}(t,y,N)$  has a smooth inverse, and its inverse is decreasing from $\dis ]\Phi_ \mathcal{F}(t,y),  +\infty[$ to $\dis ]0, N_ \mathcal{F}(t,y)[$, where 
$\dis \Phi_ \mathcal{F}(t,y) \: := \:  \mathcal{F}(t,y,N_ \mathcal{F}(t,y) < 0$.   {\em We  make a slight abuse of notation, and denote $\dis  \mathcal{F}^{-1}(t,y,\Phi)$ such inverse.} 

\medskip
Back to  \eqref{Phi0}, we obtain 
\begin{equation} \label{Phi0BL1}
 \pa^2_z \Phi^0(t,y,z) = \mathcal{X}(t,y, \Phi^0(t,y,z))  
\end{equation}
where 
\begin{equation} \label{Phi0BL2}
\mathcal{X}(t,y,\Phi) \: := \: n^0(t,y,0) \left(  \mathcal{F}^{-1}(t,y,\Phi) - e^{-\Phi} \right).
\end{equation}
This is an autonomous ODE in $z$, $(t,y)$ playing the role of parameters. Of course, the r.h.s. is only defined as long as $ \mathcal{F}^{-1}$ is.   The ODE  is completed by boundary conditions:
\begin{equation} \label{Phi0BL3}
\Phi^0(t,y,0) = \phi_{ref} + \phi_c(t,y)  - \phi^0(t,y,0) = \phi_c(t,y) + \ln\left(\frac{n^0(t,y,0)}{n_{ref}}\right), \quad \lim_{z\rightarrow +\infty} \Phi^0(t,y,z) = 0. 
\end{equation}  
We shall prove in the next paragraph existence and uniqueness of a solution to this system, under a condition of type \eqref{smallness}.

\begin{remark}
\label{remarkamplitude1}
From  \eqref{Phi0BL3}, we shall set
$$ \sup_{t \in {0, T]}, \, y \in \mathbb{R}^2}|\Phi^0(t,y,0)|= |\phi_b  +  \ln  n^0(t,y,0)|=: \delta.$$
This parameter $\delta$ measures how the quasineutral solution is far from satisfying the Dirichlet condition.
 It will determine the size of the boundary layer.
\end{remark}

\medskip
Once $\Phi^0$ has been determined, we obtain $N^0$ using the relation 
$$ \frac{n^0(t,y,0) + N^0(t,y,z)}{n^0(t,y,0)} =  \mathcal{F}^{-1}(t,y,\Phi^0(t,y,z)).  $$
In turn,  \eqref{N0U03} gives $U^0_3$. 

\medskip
Finally, we use the horizontal components of the momentum equation, that yield
$$ (\Gamma u^0_3 + U^0_3) \pa_z  U^0_y \: = \:  0 $$
which together with decay entails that 
$ U^0_y = 0$.
This concludes the formal derivation of  $(n^0,u^0,\phi^0)$, and $(N^0,U^0, \Phi^0)$. 

\bigskip
c) From there, one can derive $\dis (n^1, u^1, \phi^1)$, and  $\dis (N^1,U^1, \Phi^1)$. More generally, knowing  $(n^k, u^k, \phi^k)$, and  $\dis (N^k,U^k, \Phi^k)$, $k \le i-1$, one can derive equations on $(n^i, u^i, \phi^i)$, and then on $(N^i,U^i, \Phi^i)$. We shall stick to $i=1$ for the sake of brevity, and refer to \cite{GVHKR} for a full derivation in a close context. 
 
 \medskip
 Clearly, 
 \begin{equation} \label{u1n1}
 \left\{
 \begin{aligned}
& \pa_t n^1 \: + \: \div(n^0 u^1) \: + \: \div(n^1 u^0) \: = \: 0, \\
& \pa_t u^1 + u^1 \cdot \na u^0 + u^0 \cdot \na u^1 + T^i \left( \frac{\na n^1}{n^0} - \frac{n^1 \na n^0}{(n^0)^2} \right) \:  = \:  \na \phi^1, \\
& - e^{-\phi^0} \phi^1 = n^1. 
\end{aligned}
\right.
\end{equation}
The well-posedness of this system over $(0,T)$ (where $T$ is the time up to which the supersonic boundary condition on $u^0$ is satisfied) will be reminded in the next paragraph. 

\medskip
As regards the boundary layer terms, we look as before at the Poisson equation of \eqref{EP}, which leads to
\begin{equation} \label{Phi1}
\pa_z^2 \Phi^1 = (z \Gamma \pa_3 n^0 +  \Gamma n^1 + N^1) + e^{-(\Gamma \phi^0 + \Phi^0)} (z \Gamma \pa_3 \phi^0 + \Gamma \phi^1 + \Phi^1).
\end{equation}
Then, the mass equation  yields (remember that $U^0_y = 0$)
\begin{multline*} 
\pa_z ((\Gamma n^1 + N^1) (\Gamma u^0_3 + U^0_3)) + \pa_z \left( ( \Gamma n^0 + N^0) (\Gamma u^1_3 + U^1_3) \right) 
= F_1, \\
F_1:=  -\pa_z (z \Gamma \pa_3 n^0 U^0_3 -  N^0 z \Gamma \pa_3 u^0_3)  - {\rm div}_y( N^0 \Gamma u^0_y) - \pa_t N^0. 
\end{multline*}
Integration from $z$ to infinity  yields
\begin{multline} \label{N1U13}
(\Gamma n^1 + N^1) (\Gamma u^0_3 + U^0_3) +  (\Gamma n^0 + N^0) (\Gamma u^1_3 + U^1_3) \\
  = \int_{+\infty}^z F_1  +  \Gamma n^1\Gamma u^0_3 + \Gamma n^0\Gamma u^1_3 =: F_2.
\end{multline}
Then, we write down $\eps^0$ terms in the vertical component of the momentum equation: 
\begin{multline*} \pa_z  (( \Gamma u^0_3 + U^0_3) (\Gamma u^1_3 + U^1_3)) \: + \: T^i \pa_z \frac{\Gamma n^1 + N^1}{\Gamma n^0 + N^0} 
 = \pa_z \Phi^1 + F_3, \\
F_3:=  - \pa_z (z \Gamma \pa_3 u^0_3 U^0_3) - T^i \pa_z \left( \frac{z \pa_3 n^0}{\Gamma n^0 + N^0} - \frac{z \pa_3 n^0}{\Gamma n^0} \right)  -  \Gamma u^0_y \cdot \na_y  U^0_3  -  \pa_t U^0_3. 
\end{multline*}
One can as before integrate with respect to $z$, to get
\begin{multline*}
(\Gamma u^0_3 + U^0_3) (\Gamma u^1_3 + U^1_3) \:  + \:  T^i \frac{\Gamma n^1 + N^1}{\Gamma n^0 + N^0}   = \Phi^1 - \int_{+\infty}^z  F_3
\: + \:  \Gamma u^0_3 \Gamma u^1_3 \: + \:  T^i \frac{\Gamma n^1}{\Gamma n^0} =: \Phi^1 + F_4. 
\end{multline*}
Using \eqref{N0U03} and \eqref{N1U13} to transform the first term at the l.h.s., we end up with 
\begin{equation*}
 \left( -\frac{(\Gamma u^0_3 \Gamma n^0)^2}{(\Gamma n^0 + N^0)^3} + \frac{T^i}{\Gamma n^0 + N^0} \right)(\Gamma n^1 + N^1) = \Phi^1 + F_5
 \end{equation*}
 where 
 $$ F_5 \: :=  \: + \Black  \frac{\Gamma n^0 \Gamma u^0_3}{(\Gamma n^0 + N^0)^2} F_2 + F_4 $$
 is  known from the previous steps of the construction.
 This expression allows to express $\Gamma n^1 + N^1$ in terms of $\Phi^1$, and to go back to \eqref{Phi1} to have a closed equation on $\Phi^1$. A tedious but straightforward calculation shows then that 
\begin{equation} \label{Phi1BL}
\pa^2_z \Phi^1(t,y,z) \: = \: \pa_\Phi \mathcal{X}(t,y,\Phi^0(t,y,z)) \Phi^1(t,y,z) \: + \: F_6(t,y,z) 
 \end{equation}
 with 
 \begin{multline*}
 F_6(t,y,z) \: := \:    n^0(t,y,0) \left[ \pa_N  \mathcal{F}\left(t,y, \frac{n^0(t,y,0) + N^0(t,y,z)}{n^0(t,y,0)}\right) \right]^{-1} F_5(t,y,z) \\
  + \: n^0(t,y,0) (e^{-\Phi^0} - 1) z \pa_3 \phi^0(t,y,0) + n^0(t,y,0) \phi^1(t,y,0).
 \end{multline*}
This equation is completed with the boundary condition 
\begin{equation} \label{Phi1BL2}
\Phi^1(t,y,0) = - \phi^1(t,y,0), \quad \lim_{z\rightarrow +\infty} \Phi^1(t,y,z) = 0. 
\end{equation}  
 The well-posedness of this equation will be discussed in the next section. 
 
 \medskip
 As soon as $\Phi^1$ is determined, the expression of $N^1$ follows, and from \eqref{N1U13}, we get the expression of $U^1_3$. The horizontal part $U^1_y$ is obtained as for $U^0_y$ by considering the horizontal components of the momentum equation in \eqref{EP}. We leave all details to the reader.  As mentioned before, the next order terms in the expansion satisfy similar problems, and we omit their construction for the sake of brevity. 
 
 \subsection{Well-posedness} \label{subsecwp}
We discuss here the well-posedness of the reduced models derived formally in the previous section. 

\medskip
a) The supersonic condition \eqref{supersonic} expresses that all characteristics of the Euler system \eqref{E}
 are outgoing, so that no boundary condition is necessary for solvability. More precisely, starting from the initial data satisfying \eqref{supersonic}-\eqref{initialregularity},  there exists a small time $T > 0$ and a smooth solution \cite{BenzoniSerre}
$$ n^0 \in  n_{ref} + C^\infty([0,T], H^\infty(\R^3_+)), \quad u^0 \in u_{ref}  + C^\infty([0,T], H^\infty(\R^3_+)). $$
Moreover, one can assume that the supersonic condition is satisfied globally over $[0,T]$: 
\begin{equation} \label{supersonicbis}
 \inf n^0 > 0, \quad  \sup u^0_3  < 0, \quad ( \sup u^0_3)^2 >  T^i + 1.
 \end{equation}
 Under this last condition, one can solve the equations on $(n^i, u^i, \phi^i)$, $i \ge 1$,  over $[0,T]$. We remind that these equations are  of the type 
 \begin{equation*} 
 \left\{
 \begin{aligned}
& \pa_t n^i \: + \: \div(n^0 u^i) \: + \: \div(n^i u^0) \: = \: f^i_n, \\
& \pa_t u^i + u^i \cdot \na u^0 + u^0 \cdot \na u^i + T^i \left( \frac{\na n^i}{n^0} - \frac{n^i \na n^0}{(n^0)^2} \right) \:  = \:  \na \phi^i \: + \:  f^i_u, \\
& - e^{-\phi^0} \phi^i = n^i \: + \:  f^i_\phi  
\end{aligned}
\right.
\end{equation*}
(see \eqref{u1n1} in the special case $i=1$). These are linear hyperbolic  systems with unknowns $(n^i, u^i)$, and with outgoing characteristics over $[0,T]$ by \eqref{supersonicbis}. Starting (for instance) from zero initial data   
$$ n^i\vert_{t=0} = 0, \quad u^i\vert_{t=0} = 0, $$
these systems are shown inductively to possess unique regular solutions over $[0,T]$. More precisely, 
one can check  inductively that all source terms $f^i_{n,u,\phi}$ are in $C^\infty([0,T]; H^\infty(\R^3_+))$, so that 
$$   n^i \in  C^\infty([0,T], H^\infty(\R^3_+)), \quad u^i \in  C^\infty([0,T], H^\infty(\R^3_+)), $$
and the same for $\phi^i$. 

\medskip
b) We now have to address rigorously the construction of boundary layer systems. Note that the $N^i$'s and $U^i$'s are given in an extrinsic manner from the $\Phi^i$'s, so that we only need to focus on the latter.  We begin by  discussing  the well-posedness of \eqref{Phi0BL1}-\eqref{Phi0BL2}-\eqref{Phi0BL3}.  We already noticed that $t,y$ are just parameters,  which makes it  an ODE problem. This ODE problem has been  addressed in \cite{AMR}, see also \cite{Suz}.  Omitting the dependence with respect to $t,y$, it reduces to the search of a trajectory $z \mapsto \Phi(z)$ of a second-order   ODE  
\begin{equation} \label{EDO}
  \Phi'' = \mathcal{X}(\Phi),  \quad \mathcal{X}(\Phi) =  \mathcal{F}^{-1}(\Phi) - e^{-\Phi}, \quad  \mathcal{F}(N) =   \frac{u^0_3(0)^2}{2 N^2} \: + \: T^i \ln(N) \: - \:   \frac{u^0_3(0)^2}{2}. 
\end{equation}
such that $\Phi(0) = \Phi_0$, $\: \lim_{z \rightarrow 0} \Phi = 0$ for some constant $\Phi_0$. 
Note that we only allow solutions $\Phi$ with values in  some interval $]\Phi_ \mathcal{F}, +\infty[$, $\Phi_ \mathcal{F} < 0$,  in order for $ \mathcal{F}^{-1}$ and $\mathcal{X}$ to be defined:  see the discussion below  \eqref{defF}. 

\medskip
This problem can be solved using the hamiltonian structure of the equation. Introducing the antiderivative $\mathcal{V} = \int_0^\Phi  \mathcal{X}$, we get  
$$  (\Phi')^2 =    2 \mathcal{V}(\Phi). $$
The variations of $\mathcal{V}$ can be determined using that $\mathcal{V}'(\Phi) = \mathcal{X}(\Phi)$ has the same sign as $\mathcal{Y}(N) = N - e^{ -\mathcal{F}(N)}$, which in turn has the same sign as $\ln N +  \mathcal{F}(N)$. We refer to \cite{Suz} for more details.  As a result, there exists $\delta_0 > 0$ such that: {\em for all $\Phi_0 \in [-\delta_0, +\infty[$, there is a unique  (monotonic) solution of \eqref{EDO} connecting $\Phi_0$ to $0$}. Moreover,  as ${\mathcal V}''(0) > 0$,   $(0,0)$ is a hyperbolic point for the $2\times2$ system associated to \eqref{EDO}, which yields the exponential decay  of $\Phi$ by the stable manifold theorem.  More precisely, the rate of decay is given by 
\begin{equation}
\label{decayBL}
 \gamma:=\sqrt{\mathcal{V}''(0)} = \sqrt{\frac{T^i + 1 - u^0_3(0)^2}{T^i  - u^0_3(0)^2}}.
\end{equation}

\medskip
Back to the original problem (with dependence in $t,y$), this analysis provides a solution of \eqref{Phi0BL1}-\eqref{Phi0BL2}-\eqref{Phi0BL3} under the first inequality in  \eqref{smallness}. The regularity   of $\Phi^0$ with respect to $(t,y)$ follows from standard arguments: using  (for instance) the implicit function theorem, one can show that
$$ \mathcal{X}(t,y,\phi) = \widetilde{\mathcal{X}}\left(n^0(t,y,0), u^0_3(t,y), \Phi\right), $$
and then
$$ \Phi^0(t,y,z) =  \widetilde{\Phi^0}\left(n^0(t,y,0), u^0_3(t,y), \phi_c(t,y),z\right) $$
for smooth functions
$$\widetilde{\mathcal{X}} = \widetilde{\mathcal{X}}(n,w,\Phi), \:  \: \widetilde{\Phi^0} = \widetilde{\Phi^0}(n,w,\phi, z).$$
 Moreover,  one can check that $\widetilde{\Phi^0}(n_{ref},\cdot) = 0$. The smoothness in $(t,y)$ and $H^\infty$ integrability in $y$ follow.   
For the sake of brevity, we leave the details to the reader. 

\begin{remark}
\label{remarkamplitude2}
Note that the decay rate \eqref{decayBL} is independent of the amplitude parameter defined in Remark \ref{remarkamplitude1}.
We thus get for $\Phi^0$ an estimate under the form: there exists $C>0$ , $\gamma_{0}>0$ such that for every $\delta \in (0, \delta_{0})$,  we have the estimate
$$ |\Phi^0| \leq  C\delta e^{- \gamma_{0} z}.$$ 
\end{remark}

\medskip
As regards the next order boundary layer terms, they all (formally) satisfy equations of the type 
$$ \pa^2_z \Phi^i(t,y,z) \: = \: \pa_\Phi \mathcal{X}(t,y,\Phi^0(t,y,z)) \Phi^i(t,y,z) \: + \: F^i $$
 (see \eqref{Phi1BL} for $i=1$), where $F^i$ depends on the lower order profiles.  Their well-posedness (in the space of smooth functions,  exponentially decaying in $z$) is shown inductively: freezing $t,y$, it reduces to show the  well-posedness of  a linear ODE of the type 
$$  \pa^2_z \Phi \: = \: \mathcal{X}'(\Phi^0) \Phi  \: + \: F, $$
for some smooth exponentially decaying $F$, with boundary conditions $\Phi\vert_{z=0} = \psi, \: \Phi\vert_{z=\infty} = 0$. Note that, up to consider $\tilde \Phi(z) = \Phi(z) - \psi \chi(z) $, with $\chi \in C^\infty_c(\R_+)$ satisfying $\chi(0)=0$, we can always assume that $\psi = 0$.  

\medskip
As  $\mathcal{X}'(\Phi) < 0$   for large $\Phi$, the well-posedness of these systems is not clear under a mere lower bound on $\Phi^0$. But as $ \mathcal{X}' (0) > 0$, up to take a smaller $\delta_0$ and impose the additional upper bound in \eqref{smallness}, we ensure that $\mathcal{X}'(\phi^0) \ge \alpha > 0$ uniformly in $z$.  The existence of a variational solution (say in $H^1_0(\R_+)$)  follows then from Lax-Milgram's lemma. Smoothness in $z$ is a consequence of standard elliptic regularity results, whereas the exponential decrease follows again from the stable manifold theorem. 

\medskip
The proof of Theorem \ref{deriv} is thus complete. 
%
%

  \section{Linear Stability}
  
  In this section, we establish $L^2$ type estimates for linear systems of the following form:
  \begin{equation} \label{linear}
\left\{
\begin{aligned}
& \pa_t  \np \: + \:   (u_{a} + u )\cdot \nabla \np \: + \:  (n_a + n) \nabla \cdot \up   \: = \:  r_n,  \\
& \pa_t \up \: + \:  (u_a + u)  \cdot \na \up  \: + \:   T^i   \frac{\na \np}{n_a + n} \:  = \:  \na \fp + r_u, \\
& \eps^2 \Delta \fp \: = \:  \np \:  +  \: e^{-\phi_a} \, \fp \,  (  1 + h(\phi) ) \:  + \:  r_\phi,
\end{aligned}
\right.
\end{equation}
 where $r = (r_n, r_u,r_\phi)$ is a given source term, and where $h \in \{h_0, h_1\}$, where
 \begin{equation*}
h_0(\phi) \: := \:  -{ e^{-\phi} - 1 + \phi \over \phi}, \quad h_1(\phi) \: := \:  e^{-\phi}- 1,
\end{equation*}
  {\it cf} \eqref{h0h1}. 
We add to the system  the boundary condition
\begin{equation}
\label{BCfp}
\fp\vert_{x_3=0} = 0.
\end{equation}
  
   To prove our main linear stability estimate, we shall use 
   Goodman type weights in order to use the stabilizing effect of convection in the problem.
   From the construction of the approximate solution in the previous section, in particular, from Remarks \ref{remarkamplitude1}
    and \ref{remarkamplitude2}, we know that the main  boundary layer part of the approximate solution satisfies
     an estimate under the form 
   \begin{equation}
\label{boundna0}
  \sup_{(0,T) \times \R^2} | \pa_{x_3} (n_{a},u_{a},\phi_{a})(\cdot, x_3) |  +   \eps | \pa^2_{x_3} (n_{a},u_{a},\phi_{a})(\cdot, x_3) |\le \frac{\delta}{\eps} e^{-\frac{\gamma_{0} \, x_3}{\eps}},
\end{equation}
where $\delta$ will be assumed small whereas $\gamma_{0}$ is fixed.   We recall that as in the previous section, for any function $f=f(t,x)$, we denote by    
 $\Gamma f$ the function  $(t,y,z) \mapsto  f(t,y,0)$. We shall also assume that the trace on the boundary of the tangential velocity  of the limit system 
   is small, i.e.
    \begin{equation}
    \label{hypvitan}
    | \Gamma u^0_{1, 2} | \leq \delta.
    \end{equation}

  We   introduce the \emph{weight function} 
  $$\eta(x_3) := e^{ {\delta \over \mu^2} ( 1- e^{-\mu x_3/\eps}) }, $$ where $\mu, \gamma>0$ are some fixed parameters, to be specified later. Observe that $\eta$ satisfies the following properties:
  \begin{equation}
  \label{systeta}
  \left\{
  \begin{array}{ll}
  \eta(0)=1, \\
  \eta'(x_3) = {\delta \over \mu \eps}e^{-\mu x_3/\eps} \, \eta(x_3).
  \end{array}
  \right.
  \end{equation}
  
  The main  idea is that the parameter  $\mu>0$,   $\mu \leq \gamma_{0}$ will be chosen small enough.  We  also assume that  the parameter   $\delta$ that measures the strength of the boundary layers  is sufficiently small  such that $\delta/\mu^2$ is also small. This yields in particular
   \begin{equation}
   \label{estimpoids}
    {1 \over 2 } \leq \eta(x_{3}) \leq 2 , \quad x_{3} \eta' \leq 4, \quad \forall x_{3} \geq 0.
   \end{equation}
  For example, we can choose $\mu = \delta^{1 \over 4}$. The assumptions on $\mu$ further imply 
 \begin{equation}
\label{betagrand}
\eta(x_3) \frac{C_a \delta}{\eps} e^{-\gamma_{0} \, x_3/\eps} \leq  C_{a}\,  \mu\, \eta'(x_3), \quad \forall \, x_3>0,
\end{equation}
This inequality will be crucial in many estimates of the paper. We also note for later purposes that
\begin{equation} \label{eta''}
|\eta''| = \left| \frac{\delta}{\mu \eps} e^{-\mu x_3/\eps} \eta'(x_3) + \frac{\mu}{\eps} \eta'(x_3) \right| \: \le \:   C \frac{\mu}{\eps} \eta'(x_3)
\end{equation}

We shall eventually assume that the forcing terms $ (r_{n}, r_{u}, r_{\phi})$ can be split into a  small singular part  and a regular part:
\begin{equation}
\label{hypsource}
| (r_{n}, r_{u}, {r_{\phi } \over \eps}, \eps \nabla r_{n}, \eps \nabla r_{u}, \partial_{t} r_{\phi}) | \leq \mu \eta' R^s + R^r.
\end{equation}
Note that since the derivative of $\eta$ is of order $1/\eps$, the first term in the r.h.s. of \eqref{hypsource} is singular in $\eps$. Concretely, the linearized system \eqref{linear} will be obtained by taking $\eps$-derivatives of \eqref{approxEP}. The remainders will contain commutators, and satisfy \eqref{hypsource}.

  The crucial  weighted $L^2$ estimate is given in the following proposition.
  
  \begin{proposition}
\label{propL2}
 Let $(n_{a}, u_{a},\phi_{a})$ be the approximate solution constructed in Theorem \ref{deriv}
  and   consider some smooth $(n, u, \phi)$ on $[0, T]$ such that for some $M>0$
  \begin{multline}
  \label{hypmin}
 n_{a} + n \geq  1/M, \quad   e^{-\phi_{a}}(1 + h(\phi)) \geq  1/M, \\  {\|(n, u, \phi)\|_{L^\infty} \over \eps}+ \|\nabla (n,u, \phi) \|_{L^\infty}+  \|\partial_{t}(n, \phi)\|_{L^{\infty}} \leq M ,  \,  \forall t \in [0,T], \, x \in \mathbb{R}^3_{+}
  \end{multline} 
  and such that the Bohm condition is verified on the boundary
  \begin{equation}
  \label{hypbohm}
  \Gamma (u_{a} + u)_{3} < - {1 \over M}, \quad\big( \Gamma(u_{a} + u)_{3} \big)^2 \geq T^i+ 1 + {1 \over M}.
  \end{equation}
  Moreover, let us assume that the source term $(r_{n}, r_{u}, r_{\phi})$ of \eqref{linear} verifies the assumption \eqref{hypsource}.
  Then,  there exists $\delta_{0}>0$, $\mu_{0}>0$  and  $ C( C_{a}, M )$    ($C_{a}$ only depends on the approximate solution)  such that for every $\eps \in (0, \eps_{0}]$, 
    for  every $\delta \in (0, \delta_{0}]$ and for every $\mu \in (0, \mu_{0}]$ with $\delta/\mu^2$ sufficiently small, we have   for the solution
    $(\np, \up , \fp)$ of \eqref{linear}
    the estimate
 \begin{align*}
 & \sqrt{\mu} \Big( \big\| \big( \np,\up, \fp, \eps \nabla \np, \eps \nabla \up,\eps \na \fp \big) \big\|_{L^\infty_{T}L^2(\R^3_{+})}^2
  +  \|\sqrt{\eta'}( \np, \up, \fp, \eps \nabla \np, \eps \nabla \up, \eps \nabla \fp)\|_{L^2_{T}L^2(\R^3_{+})}^2  \\
  & \quad \quad \quad \quad \quad+ 
   \|\Gamma( \np, \up, \fp, \eps \nabla \np, \eps \nabla \up, \eps \partial_{3} \fp)\|_{L^2_{T}L^2(\R^2)}^2  \Big)
   \\
 & \quad \quad  \leq   C(C_{a}, M)\Big( \big\|\big( \np_0,\up_0, \fp_0, \eps \nabla \np_0, \eps \nabla \up_0, \eps \nabla \up_0,\eps \na \fp_0) \big\|_{L^2(\R^3_{+})}^2  + \|(\np, \up, \eps \nabla \np, \eps \nabla \up , \fp, \eps \nabla \fp)\|_{L^2_{T}L^2(\mathbb{R}^3_{+})}^2    \\
& \quad \quad \quad \quad \quad \quad  \quad \quad \quad+ \big\|   R^r\big\|_{L^2_{T}L^2(\R^3_{+})}^2
 + \mu \| \sqrt{\eta'} R^s \|_{L^2_{T}L^2(\R^3_{+})}^2 \Big).
  \end{align*}

\end{proposition}

Here $L^2_T$ and $L^\infty_T$ stand for $L^2([0,T])$ and $L^\infty([0,T])$. This whole section is dedicated to the proof of Proposition~\ref{propL2}.

\begin{proof}[Proof of Proposition~\ref{propL2}] We first gather some useful bounds satisfied by the approximate solution $(n_a,u_a,\phi_a)$.
In the proof, we shall denote by $C_{a}$ and $C(C_{a}, M)$  numbers which depends only on the estimates of the approximate solution and on
 the number $M$ defined in \eqref{hypmin}  and that may change from line to line. The important thing is  that they are  uniformly bounded for
 $\eps \in (0,1]$,  $\delta \in (0, 1)$ and  $T \in (0, T_{0}]$ where $T_{0}$ is the interval of time on which the approximate solution is defined.
We first have, by construction (see Theorem~\ref{deriv}):
\begin{equation} \label{boundnuf}
 \sup_{(0,T) \times \R^3_+}  | \na_{x_1,x_2,t}^k(n_{a},u_{a},\phi_{a}) | \le C_a, \quad C_a > 0,
 \end{equation}
 with $C_{a}$ independent of $\eps$.
 In the other hand, we have for all $x_3>0$:
 \begin{equation}
\label{boundna}
  \sup_{(0,T) \times \R^2} |  \eps^l\pa_{x_3}^{1+l}  \nabla_{x_{1}, x_{2}, t}^k(n_{a},u_{a},\phi_{a})(\cdot, x_3) | \le \frac{\delta}{\eps} e^{-\frac{\gamma_{0} \, x_3}{\eps}} + C_{a},
\end{equation}
where we recall that  $0<\delta \ll 1$ can be considered as a small  parameter.
 \bigskip

\bigskip

We shall combine many energy estimates for the proof of Proposition \ref{propL2}. As already pointed out in the introduction, our approach shares features with  the analysis led in \cite{NOS}.  The starting point of all the energy estimates will be the following lemma:
 \begin{lemma}
 \label{lem+}
 Under the assumptions of Proposition \ref{propL2}, we have the estimate
 \begin{multline*} {1 \over 2 }{d \over dt} \int_{\mathbb{R}^3_{+}} \eta \left( (n_{a}+ n ) {|\up|^2 \over 2 }  + { |\np|^2 \over  n_{a}+ n} \right)\, dx 
  +{1 \over 2} \int_{\mathbb{R}^3_{+}}  \eta'\, Q^0(\np, \up) \, dx + { 1 \over 2 } \int_{x_{3}= 0} Q^0(\np, \up) - \I \\
   \leq C(C_{a}, M)\Big( \|(\np, \up )\|_{L^2}^2 +  \|R^r \|_{L^2}^2   + \mu \int_{\R^3_{+}}
    \eta' | R^s|^2\Big)
   \end{multline*}
   where the quadratic form $Q^0$ is associated to the  symmetric matrix 
     $$\mathcal{Q}^0 \: : = \: \left( \begin{array}{cc} T^i \Gamma( {|(u_{a} + u)_{3}|\over n^a+ n}) & -T^i  e^\top\\  -T^{i} e & \Gamma\big( (n_{a}+ n)  |(u_{a}+ u)_{3}| \big) \mbox{Id}_{3}\end{array} \right), \quad e^\top= (0,0, 1)^\top$$
   and
\begin{equation} \label{calI}
 \I \: := \: \int_{\R_{3}^+} \eta \nabla \fp\cdot (n_{a}+ n) \up \, dx.
 \end{equation}
 \end{lemma} 
 The main difficulty in proving Proposition \ref{propL2} will be to handle the term
 $\I$, that involves the potential $\fp$.
We note that in the left hand side of the above estimate, the quadratic form $Q^0$ is positive thanks to the Bohm condition \eqref{hypbohm}. 
\subsubsection*{Proof of Lemma \ref{lem+}}

First, multiplying the velocity equation by $(n_a + n) \, \up \, \eta$, and performing standard manipulations, we obtain:
\begin{align}
 \nonumber&  {{d}\over{dt}} \int_{\R^3_+}  \eta \, (n_a + n) \frac{|\up|^2}{2}  = \int_{\R^3_+} \eta \,  (n_a + n) \up \cdot \partial_t \up + \int_{\R^3_+} \eta \,  \pa_t (n_a + n) \frac{|\up|^2}{2} \\
&=  \I \: + \:  I_1 \: + \:  I_2  \: + \: \int_{\R^3_+}  \eta \,  \left( r_u \cdot ((n_a+n) \up) \right) + \int_{\R^3_+}  \eta \,  \pa_t (n_a + n) \frac{|\up|^2}{2},  
\end{align}
where 
$$ I_1 \: :=  \: -  T^i  \int_{\R^3_+} \eta \,   \frac{\na \np}{n_a + n}  \cdot ((n_a + n) \up), \quad 
 I_2 \: := \: - \int_{\R^3_+} \eta \,  [(u_a + u) \cdot \na \up] \cdot (n_a + n) \up. $$
We recall that $\I$ was defined in \eqref{calI}.  The last two  terms at the r.h.s. of \eqref{momentum} can be  easily estimated through \eqref{boundnuf}, 
 \eqref{hypmin}, \eqref{estimpoids} and \eqref{hypsource}. We find
\begin{multline}  \label{momentum}
{{d}\over{dt}} \int_{\R^3_+}  \eta \, (n_a + n) \frac{|\up|^2}{2}  \:  \le \:  \I \: + \:  I_1 \: + \:  I_2  \\
 + \:  C(C_{a}, M)\Big(  \|    R^r(t)\|_{L^2(\R^3_+)}^2 + \mu   \|  \sqrt{\eta'}  R^s(t)\|_{L^2(\R^3_+)}^2 +
  \| \up(t) \|_{L^2(\R^3_+)}  + \mu \| \sqrt{\eta'} \dot u \|_{L^2(\R^3_{+})}^2  \Big).
\end{multline}

 \medskip
 
 {\bf 1) Treatment of $I_1$.} Integrating by parts, we have the identity
\begin{equation*}
\frac{1}{T^i} I_1  \: = \: -   \int_{\R^3_+}  \eta \, \na \np \cdot \up = \int_{\R^3_+} \eta \, \np \, \div \up \: + \:  \int_{\R^3_+}  \eta'  \np \, \up_3 
 \: + \:   \int_{x_3= 0} \np \,  \up_3.
\end{equation*}
 To write the boundary term,  we have used that  $\eta (0)= 1$, see \eqref{systeta}. We can then use the evolution equation on $\np$ to express $\div \up$ in terms of $\np$.  We find 
\begin{align}
\nonumber  \frac{1}{T^i} I_1  \: & = \: - \int_{\R^3_+} \frac{\eta}{n_a +n} \left( \pa_t + (u_a + u) \cdot \na \right) \frac{|\np|^2}{2} + \int_{\R^3_+} \frac{\eta}{n_a + n} r_n \, \np \: \: + \:  \int_{\R^3_+}  \eta'  \np \, \up_3 \: + \:   \int_{x_3= 0} \np \,  \up_3  \\
\nonumber  & =   - \pa_t   \int_{\R^3_+}  \frac{\eta}{n_a +n}  \frac{|\np|^2}{2} \: + \: \int_{\R^3_+} \eta' \frac{(u_a + u)_3}{n_a + n}   \frac{|\np|^2}{2} \: + \: \int_{x_3 = 0} \frac{(u_a + u)_3}{n_a + n}   \frac{|\np|^2}{2} \\
\label{estimI1}   & \:  + \:  \int_{\R^3_+}  \eta'  \np \, \up_3 
 \: + \:   \int_{x_3= 0} \np \,  \up_3  \: + \: J_1 \:+ \:  J_2
\end{align}

  with 
\begin{equation*}
 J_1 \: := \: \int_{\R^3_+} \div \left( \frac{u_a + u}{n_a + n} \right)  \frac{|\np|^2}{2}, \quad
 J_2 \: := \:   \int_{\R^3_+}  \pa_t \left( \frac{1}{n_a + n} \right) \frac{|\np|^2}{2} \: + \: \int_{\R^3_+} \frac{\eta}{n_a + n} r_n \, \np. 
 \end{equation*}
 For the last term $J_2$,  we  use  again \eqref{boundnuf}, 
 \eqref{hypmin}, \eqref{estimpoids} and \eqref{hypsource} to get
\begin{equation} \label{estimJ2}
  J_2 \: \leq \: C(C_{a}, M) \Big( \| R^r(t)\|_{L^2}^2 +  \| \,\np(t)\|_{L^2}^2 + \mu \| \sqrt{\eta'} R^s(t) \|_{L^2}^2
 + \mu \| \sqrt{\eta'} \np(t) \|_{L^2}^2\Big).
 \end{equation}
To bound $J_1$, we use \eqref{hypmin} and \eqref{boundna0} to state
\begin{align}
\nonumber
J_1 & \: \le \: \int \eta \, C(C_a,M) \left(\frac{\delta}{\eps} e^{-\gamma_0 x_3/\eps} + 1 \right) \, |\np|^2 \\
\label{estimJ1} 
& \: \le \: C(C_a,M) \left( \mu \int_{\R^3_+} \eta' |\np|^2  \: + \: \| \np \|^2_{L^2} \right),
\end{align}
where we have used \eqref{betagrand} to go from the first to the second line.   We insert \eqref{estimJ2}-\eqref{estimJ1} in \eqref{estimI1} to obtain 
\begin{equation} \label{estimaI1}
\begin{aligned}
 \frac{1}{T^i} I_1  \: & \le \: - \pa_t   \int_{\R^3_+}  \frac{\eta}{n_a +n}  \frac{|\np|^2}{2} \: + \: \int_{\R^3_+} \eta' \frac{(u_a + u)_3}{n_a + n}   \frac{|\np|^2}{2} \: + \: \int_{x_3 = 0} \frac{(u_a + u)_3}{n_a + n}   \frac{|\np|^2}{2} \: + \:  \int_{\R^3_+}  \eta'  \np \, \up_3   \\
  &  \: + \:   \int_{x_3= 0} \np \,  \up_3 \: + \: C(C_a,M) \left(  \mu \int_{\R^3_+} \eta' |\np|^2 \: + \: \| \np \|^2_{L^2}  \: + \:   \| R^r(t)\|_{L^2}^2 + \mu \| \sqrt{\eta'} R^s(t) \|_{L^2}^2 \right).
 \end{aligned}
 \end{equation}

  {\bf 2) Treatment of $I_2$.} We write
\begin{align*}
I_2 &= - \int_{\R^3_+} \eta \, (n_a+n)(u_a + u) \cdot \na \frac{|\up|^2}{2} \\ 
&= \int_{\R^3_+}  \eta \, \div( (n_a+n)(u_a + u)) \frac{|\up|^2}{2} +  \int_{\R^3_+}  \eta' \,   (n_a+n)(u_a + u)_3 \frac{|\up|^2}{2} +  \int_{x_3=0} (n_a+n)(u_a + u)_3 \frac{|\up|^2}{2}.
\end{align*}
Relying once again on \eqref{boundna0}-\eqref{boundnuf}, and \eqref{betagrand}, we infer that:
\begin{equation}
\label{estimI2}
\begin{aligned}
I_2 & \leq   \int_{\R^3_+}  \eta' \,   (n_a+n)(u_a + u)_3 \frac{|\up|^2}{2} +  \int_{x_3=0} (n_a+n)(u_a + u)_3 \frac{|\up|^2}{2} \\ 
& \: + C(C_a, M)  \left(\mu \int_{\R^3_+} \eta' |\up|^2   + \| \up \|_{L^2(\R^3_+)}^2 \right).
\end{aligned}
\end{equation}
\medskip
 
 {\bf Conclusion.}  Gathering \eqref{momentum}, \eqref{estimaI1} and \eqref{estimI2}, we obtain
 \begin{align*} 
 & {{d}\over{dt}} \int_{\R^3_+}  \eta \, (n_a + n) \frac{|\up|^2}{2} \: + \:    {{d}\over{dt}}  \int_{\R^3_+}  \eta  \frac{ T^{i} \np^2}{2 (n_a + n) } 
 \: \leq \:   \\
 & \: \leq \:  { {1 \over 2} T^i \int_{\R^3_+}  \eta' \, \frac{\np^2 \, ( u^a + u)_3  }{n_a + n}} + { T^i \int_{\R^3_+}  \eta'  \np \, \up_3} 
 +  \: \int_{\R^3_+}  \eta' \,   (n_a+n)(u_a + u)_3 \frac{|\up|^2}{2}   \\
& \: + \: T^i \int_{x_3 = 0} \frac{(u_a + u)_3}{n_a + n}   \frac{|\np|^2}{2} \: + \:  T^i\int_{x_3= 0} \np \,  \up_3     \: + \: \int_{x_3=0} (n_a+n)(u_a + u)_3 \frac{|\up|^2}{2} \\
& \: + \:   C(C_a, M) \Big(  \| ( \np, \dot u )\|_{L^2}^2 +   \|R^r \|_{L^2}^2  +  \mu \, 
    \int_{\R^3_{+}}\eta' \,  ({\np^2} + |\up|^2 + | R^s|^2) \Big).
 \end{align*}
 To conclude, it suffices to observe   that the only significant contribution in the terms involving $\eta'$ comes from the traces of the coefficients. Indeed, we have  that
 \begin{equation}
 \label{app0} u_a= u^0 + O( \eps + \delta), \, n_a = n^0+ O(\eps + \delta)
 \end{equation}
  and that
  \begin{equation}
  \label{subtrace} |(u^0, n^0)-  \Gamma (u^0, n^0) | \leq C_{a} x_{3}, \quad  |(u, n)-  \Gamma (u, n) | \leq C(M) x_{3}
  \end{equation}
   to obtain
 \begin{align}
 \label{subtrace2}   {(u^a + u)_{3} \over n_{a}+ n} & =   \Gamma ( {(u^0+u)_{3} \over n^0 + n }) + O(\delta + \eps + x_{3}) 
 = \Gamma ( {(u_{a}+u)_{3} \over n_{a}+ n }) + O(\delta + \eps + x_{3})  
 , \\
  \label{subtrace3}   (n_a+n)(u_a + u)_3  & =   \Gamma\big(
 (n^0 + n) (u^0 + u)_{3}\big) + O( \eps + \delta + x_{3}) \\
\nonumber  & =   \Gamma\big((n_{a} + n) (u_{a} + u)_{3}\big) + O( \eps + \delta + x_{3})
 .\end{align}
 Since $\eps \eta'$  and $x_{3} \eta'$ are uniformly bounded (see \eqref{estimpoids}), we end up with 
  \begin{multline*} {1 \over 2 }{d \over dt} \int_{\mathbb{R}^3_{+}} \eta \left( (n_{a}+ n ) {|\up|^2 \over 2 }  + { |\np|^2 \over  n_{a}+ n} \right)\, dx 
  +{1 \over 2} \int_{\mathbb{R}^3_{+}}  \eta'\, Q^0(\np, \up) \, dx + { 1 \over 2 } \int_{x_{3}= 0} Q^0(\np, \up) - \I \\
   \leq C(C_{a}, M)\Big( \|(\np, \up )\|_{L^2}^2 +  \|R^r \|_{L^2}^2   + \mu \int_{\R^3_{+}}
    \eta' | R^s|^2\Big) +   (\mu + \delta)  \,   \int_{\R^3_{+}}\eta' \,  ({\np^2} + |\up|^2). 
   \end{multline*}
   By the Bohm condition, the quadratic form $Q_0$ is positive definite : this allows to absorb the last term at the r.h.s. for $\mu$ and $\delta$ small enough. The estimate of Lemma \ref{lem+} follows.

  \bigskip
  
 {\bf A. The estimate for the weighted  physical energy}
 
 We now use Lemma \ref{lem+} in order to prove Proposition \ref{propL2}.
 The first estimate that we shall establish is the following
 \begin{multline}
\label{Afinal}
 \|(\np, \up, \fp, \eps \nabla \fp)\|_{L^\infty_{T}L^2(\mathbb{R}^3_{+})}^2 +   \| \sqrt{\eta'}(\np, \up, \fp, \eps \nabla \fp) \|_{L^2_{T} L^2(\mathbb{R}^3_{+})}^2 
  + \|\Gamma ( \np,  \up, \eps \partial_{x_{3}} \fp )\|_{L^2_{T}L^2(\mathbb{R}^2)}^2 \\
   \leq C(C_{a}, M) \Big( \|(\np, \up, \fp, \eps \nabla \fp)(0)\|_{L^2(\mathbb{R}_{+}^3)}^2 +  \| \sqrt{ \eta'} \,( \eps\, \nabla \np, \eps\,  \div \up)\|_{L^2_{T}L^2(\mathbb{R}_{+}^3)}^2
    \\+  \|(\np, \up, \fp, \eps \nabla \fp)\|_{L^2_{T}L^2(\mathbb{R}^3_{+})}^2
     + \| R^r \|_{L^2_{T}L^2(\mathbb{R}^3_{+})}^2 + \mu \| \sqrt{\eta'} R^s \|_{L^2_{T}L^2(\R^3_{+})}^2 
    \Big).
\end{multline} 
Note that the above estimate is not enough by itself because of the term 
$  \| \sqrt{ \eta'} \,( \eps\, \nabla \np, \eps\,  \div \up)\|_{L^2_{T}L^2(\mathbb{R}^3_{+})}^2$ in the right hand side that is not controlled by the left hand side.

  The first step in the proof of \eqref{Afinal} consists in manipulating the integral $\I$ (see Lemma \ref{lem+}) : it will give us some control on the potential $\fp$.
  More precisely, we first integrate by parts to write: 
\begin{equation*} 
\I = -  \int_{\R^3_{+}} \eta \,  \fp \, \up \cdot \na (n_{a}+ n)  - \int_{\R^3_{+}} \eta \,  \fp \, (n_a + n) \div \up  - \int_{\R^3_{+}}  \eta' \, \fp \,  \big( (n_{a}+ n)  \up_3 \big).
\end{equation*} 
There is no boundary term since $\fp$ satisfies an homogeneous Dirichlet condition (see \eqref{BCfp}).
The first term at the r.h.s can be controlled thanks to \eqref{boundna0}-\eqref{betagrand} and \eqref{hypmin}. We get  
\begin{equation} \label{estimcalIA}
\I  \: \le \: I_1 \: - \:  \int_{\R^3_{+}}  \eta' \, \fp \,  \big( (n_{a}+ n)  \up_3 \big)  \: + \: C(C_a,M) \, \bigl( \mu \left( \| \sqrt{\eta'} (\up,\fp) \|_{L^2}^2 +   \| (\up,\fp) \|_{L^2}^2 \right) 
\end{equation}
with $I_1 \: := \:   - \int_{\R^3_{+}} \eta \,  \fp \, (n_a + n) \div \up$. As above, we use the evolution equation on $\np$ to express $(n_a + n) \div \up$ in terms of $\np$ : 
\begin{equation*}
I_1 \:   := \:   \int_{\R^3_+} \eta \fp \pa_t \np \: + \: \int_{\R^3_+} \eta (u_a + u) \cdot \na \np \, \fp \: - \: \int_{\R^3_+} \eta r_n \fp.
\end{equation*}
We integrate by parts the second integral, and obtain 
\begin{equation} 
 \label{estimI1A}
I_1 \:  \le \: J_1 + J_2 + J_3 \: + \: C(C_a,M) \left( \| R_r \|_{L^2}^2 + \mu \| \sqrt{\eta'} R^s \|_{L^2}^2 \: + \:    \mu (\| \sqrt{\eta'} (\fp,\np) \|_{L^2}^2  \: + \: \| (\fp,\np) \|_{L^2}^2  \right) 
\end{equation}
with 
$$J_1 \: := \:  \int_{\R^3_+} \eta \fp \pa_t \np, \quad J_2  \: := \: -  \int_{\R^3_+} \eta (u_a +u) \cdot \na \fp \, \np  \quad J_3 \: := \: - \int_{\R^3_+} \eta' (u_a +u)_3 \fp \, \np.  $$
Note that we have used implicitly the bound 
$$ - \int_{\R^3_+} \eta \, \div(u_a + u) \fp \, \np \: \le \: C(C_a,M) \, \left(  \mu (\| \sqrt{\eta'} (\fp,\np) \|_{L^2}^2  \: + \: \| (\fp,\np) \|_{L^2}^2  \right) $$
to handle the last term coming from the integration by parts. 

\medskip
{\bf 1) Treatment of $J_1$.}  At first, by differentiating with respect to time  the Poisson equation in \eqref{linear}, we can express $\pa_t \np$ in terms of $\fp$, and substitute inside the expression for $J_1$:
\begin{align}
\nonumber J_1 & =  \int_{\R^3_+} \eta \,   \fp \, \left( \eps^2 \pa_t \Delta \fp  - \pa_t (e^{-\phi_{a}}( 1+ h(\phi)) \fp)  - \pa_t r_\phi \right) \\
\label{I21bis} &= - \eps^2  \frac{d}{dt} \int_{\R^3_+}  \frac{1}{2}\eta \,  |\na \fp|^2 \: - \:  \frac{d}{dt}  \int_{\R^3_+} \frac{1}{2}\eta \,  \fp^2 e^{-\phi_{a}} (1 + h(\phi))\\
\nonumber &\quad \: - \:  \frac{1}{2} \int_{\R^3_+} \eta \,  \pa_t (e^{-\phi_{a}}(1+h(\phi))\big) |\fp|^2  \: - \: \int_{\R^3_+} \eta \, \fp \,  \pa_t r_\phi\\
\nonumber &\quad  \: - \: \eps^2 \int \eta'  \fp \,  \pa_t \pa_{x_3} \fp.
 \end{align}
 
One has by using \eqref{boundnuf} and \eqref{hypmin} the straightforward estimate:
 $$
  - \frac{1}{2} \int_{\R^3_+} \eta \, \pa_t (e^{-\phi_{a}}(1+h(\phi))) |\fp|^2   
  \leq C(C_a, M)
   \| \fp \|_{L^2(\R^3_+)}^2
 $$
 and by using  also \eqref{hypsource}
 $$
  - \: \int_{\R^3_+} \fp \,  \pa_t r_\phi   \leq C(C_{a}, M) \big(  \| R_r \|_{L^2}^2 + \mu \| \sqrt{\eta'} R^s \|_{L^2}^2 \: + \:    \mu (\| \sqrt{\eta'} \fp \|_{L^2}^2  \: + \: \| \fp \|_{L^2}^2   \big).
$$ 
 As regards the last term, we claim 
 \begin{lemma} \label{lemmafp}
 The following inequality holds.
 \begin{multline}
\label{easybound2}
\left| \int \eps \eta'  \fp \,  \eps \pa_t \pa_{x_3} \fp\right| \: \leq \:   \| \sqrt{\eta'}\fp \|_{L^2}  \| \sqrt{\eta'} (\eps \, \na \np, \eps \, \div \up) \|_{L^2}^2
   \\+  C(C_{a}, M) \big( \|R^r\|_{L^2}^2 + \mu \| \sqrt{\eta'} R^s\|_{L^2}^2 +  \mu \| \sqrt{\eta'} \fp \|_{L^2}^2 \: + \:   \|\fp  \|^2_{L^2} \big).
\end{multline}
 \end{lemma}
 We postpone the proof of the lemma to the end of the section. Combining all the previous estimates, we deduce: 
 \begin{equation}   \label{estimJ1A}
 \begin{aligned}
J_1  & \le \:  - \eps^2  \pa_t \int_{\R^3_+}  \frac{1}{2}\eta \,  |\na \fp|^2 \: - \: \eps \pa_t \int_{\R^3_+} \frac{1}{2}\eta \,  \fp^2 e^{-\phi_{a}} (1 + h(\phi))
 \: + \: C(C_{a}, M) \big(  \|R^r\|_{L^2}^2 + \mu \| \sqrt{\eta'} R^s\|_{L^2}^2   \\
& + \:  \mu \| \sqrt{\eta'} (\np, \up, \fp) \|_{L^2}^2  \: +   \| \sqrt{\eta'}\fp \|_{L^2}  \| \sqrt{\eta'} (\eps \, \na \np, \eps \, \div \up) \|_{L^2}
 +   \|(\np,\fp)\|^2_{L^2} \big) .
\end{aligned}
\end{equation}

 \medskip
{\bf 2) Treatment of $J_2$.} 
We use  the Poisson equation to  express $\np$ in terms of $\fp$. We get  
 \begin{align}
 \nonumber J_2 & = - \int_{\R^3_+}   \eta \,   \na \fp \cdot (\eps^2 \Delta \fp (u_a + u))  \: + \:  \int_{\R^3_+}    \eta \,  \na \fp \cdot (e^{-\phi_{a}}(1+ h(\phi)) \fp (u_a+u)) \\
\label{Jdef} & + \:      \int_{\R^3_+}  \eta \,   \na \fp \cdot (r_\phi (u_a + u))  := K_1 + K_2 + K_3.
\end{align}
One has 
\begin{align}
\nonumber
K_2 &=  \frac{1}{2} \int_{\R^3_+}    \eta \,   \na |\fp|^2 \cdot (e^{-\phi_{a}} (1+ h(\phi)) (u_a+ u)) \\
\nonumber
& =  - \frac{1}{2}  \int_{\R^3_+}  \eta \,    |\fp|^2 \, \div\big(e^{-\phi_{a}}(1+ h(\phi))  (u_a + u)\big)
- \frac{1}{2}  \int_{\R^3_+}  \eta' \,    |\fp|^2 \, e^{-\phi_{a}}(1+ h(\phi))  (u_a + u)_3 \\
\label{estimK2A}
& \le \:  C(C_a,M) \, \left( \mu \|\sqrt{\eta'} \fp\|_{L^2}^2 \: + \: \| \fp \|_{L^2}^2 \right)  \: - \: \frac{1}{2}  \int_{\R^3_+}  \eta' \,    |\fp|^2 \, e^{-\phi_{a}}(1+ h(\phi))  (u_a + u)_3 
\end{align}
Again, the bound on the first term at the right hand side is a consequence of  \eqref{betagrand}.
 Thanks to \eqref{boundnuf}  and \eqref{hypsource} one has also 
\begin{align*}
K_3 \:  & \le \: C(C_{a}, M) \int_{\R^3_{+}} \big| { r_{\phi} \over \eps} \big| \| \eps \nabla \fp\| \, dx  \\
 & \: \leq   C(C_a, M) \big(   \|  R^r\|_{L^2(\R^3_+)}^2  \: + \:  \mu \|\sqrt{\eta' } R^s \|_{L^2}^2 \: + \: \mu  \| \sqrt{\eta'} \, \eps \nabla \fp \|^2_{L^2} \: + \:    \|  \e \na \fp \|^2_{L^2}
 \big). 
\end{align*}
Finally, we compute 
\begin{equation}
\label{J1}
\begin{aligned}
K_1 & = \: \eps^2 \int_{\R^3_+} \eta \, ((\na \fp \cdot \na )  (u_a + u))\cdot \na \fp \: + \: \eps^2 \int_{\R^3_+}  \eta \,  \left( \na \frac{|\na \fp|^2}{2} \right) \cdot (u_a+u)  \\
&+ \eps^2 \int_{\R^3_+} \eta' \,  \pa_3 \fp \,  \na \fp \cdot   (u_a + u)  + \eps^2 \int_{x_3=0}  \, \pa_3 \fp  \na \fp \cdot   (u_a + u) \\
& =: L_1 + L_2 + L_3 \: + \:   \eps^2 \int_{x_3=0}  \, |\pa_3 \fp|^2 (u_a + u)_3.
\end{aligned}
\end{equation}
We point out  the simplification of the boundary term, due to the homogeneous boundary condition on $\fp$. 
Once again,  \eqref{betagrand} leads to 
\begin{align*}
|L_1| &  \leq     C(C_{a}, M) \int_{\R^3_+}\eta \,{\delta \over \eps}e^{- {\gamma_{0} x_{3}\over \eps}} \,  \eps^2| \na \fp |^2  +  C(C_{a}, M) \|  \eps \nabla \dot \phi\|_{L^2}^2 \\
&  \leq   C(C_{a}, M) \Big( \mu  \| \sqrt{\eta'} \, \eps \nabla \fp \|^2_{L^2} \: + \:    \|  \e \na \fp \|^2_{L^2}\Big).
\end{align*}
For $L_2$, by integration by parts, we can write:
\begin{equation}
L_2= - \eps^2 \int_{\R^3_+}  \eta' \,  \left( \frac{|\na \fp|^2}{2} \right)  (u_a+u)_3 -\eps^2 \int_{\R^3_+}  \eta \,  \left(  \frac{|\na \fp|^2}{2} \right) \div (u_a+u) 
- \eps^2 \int_{x_3=0}   \, \left(  \frac{|\pa_{x_{3}} \fp|^2}{2} \right)  (u_a+u)_3.
\end{equation}
where we have used that $\nabla_{1,2} \fp=0$ on the boundary.
The first  term has a ``bad sign'' but will be compensated by another contribution (coming from $J_3$). 
The second term satisfies:
$$
\left| \eps^2 \int_{\R^3_+}  \eta \,  \left(  \frac{|\na \fp|^2}{2} \right) \div (u_a+u)\right| \: \leq  \: C(C_{a}, M) \big( \mu  \| \sqrt{\eta'} \, \eps \nabla \fp \|^2_{L^2} \: + \:    \|  \e \na \fp \|^2_{L^2}\big).
$$
Finally, the third (boundary) term of $L_2$ has also a bad sign,  but will be  compensated by  the  other boundary term  of \eqref{J1}.

For $L_3$, we rely on the fact that
$$ |(u^a + u )_{1,2} -  (\Gamma u^0)_{1,2}|\leq C_{a} ( \eps + \delta + x_{3}) + \|u \|_{L^\infty}.$$
 Together with our assumption   \eqref{hypvitan} on the trace ($|(\Gamma u^0)_{1,2}|\leq \delta$), this implies
\begin{align*}
L_3 \: \leq  \: \int_{\R^3_+} \eta' \,  |\eps \, \pa_3 \fp|^2 \,   (u^a(t,x) + u)_{3}  + \mu  \int_{\R^3_+} \eta' \,  |\eps \,  \na \fp|^2  + { \| u \|_{L^\infty} \over \eps} \| \eps \nabla \fp \|_{L^2}^2.
\end{align*}
 Note that the last term in the right hand side of the above estimate is bounded by $C(M)  \| \eps \nabla \fp \|_{L^2}^2$ thanks to \eqref{hypmin}.
 
\medskip
Putting together the estimates on $L_1$, $L_2$ and $L_3$, we find 
\begin{equation} \label{estimK1A}
\begin{aligned}
K_1 \: 
& \le \: - \eps^2 \int_{\R^3_+}  \eta' \,  \frac{|\na \fp|^2}{2}   (u_a+u)_3 \: + \:  \frac{\eps^2}{2} \int_{x_3=0}  \, |\pa_3 \fp|^2 (u_a + u)_3 \\
& + 
\int_{\R^3_+} \eta' \,  |\eps \, \pa_3 \fp|^2 \,   (u^a(t,x) + u)_{3} \: + \: C(C_a, M) \bigl( \mu  \| \sqrt{\eta'} \, \eps \nabla \fp \|^2_{L^2} \: + \:    \|  \e \na \fp \|^2_{L^2} \bigr).
\end{aligned}
\end{equation}
Eventually, we derive an inequality on $J_2$:
\begin{equation} \label{estimJ2A}
\begin{aligned}
J_2 \:  \le \: & - \eps^2 \int_{\R^3_+}  \eta' \,  \frac{|\na \fp|^2}{2}   (u_a+u)_3 \: + \:  \frac{\eps^2}{2} \int_{x_3=0}  \, |\pa_3 \fp|^2 (u_a + u)_3  \\
&   \: - \: \frac{1}{2}  \int_{\R^3_+}  \eta' \,    |\fp|^2 \, e^{-\phi_{a}}(1+ h(\phi))  (u_a + u)_3 \:  
\: + \: \int_{\R^3_+} \eta' \,  |\eps \, \pa_3 \fp|^2 \,   (u^a(t,x) + u)_{3} \: \\
& + \: C(C_a, M) \bigl( \| R^r \|_{L^2}^2  \: + \: \mu \| \sqrt{\eta'} \,  R^s \|_{L^2}^2  \: + \:   \| \eps \nabla \dot \phi\|_{L^2}^2  \: + \:   \mu \| \sqrt{\eta'} \, \eps \nabla \fp \|^2_{L^2}  \bigr).
\end{aligned}
\end{equation}

\medskip
{\bf 3) Treatment of $J_3$.}  First, we use the Poisson equation to replace $\np$ :
\begin{equation}
 J_3 = -  \int_{\R^3_+}  \eta' \,    \fp   (u_a+ u)_3 \left[\eps^2  \Delta \fp  -  e^{-\phi_{a}}( 1+ h(\phi)) \fp  -  r_\phi \right] =: K_1 + K_2 + K_3.
\end{equation}
Let us start with the estimate of $K_1$. With the help of an integration by parts, we get:
\begin{equation}
\begin{aligned}
K_1 &=  \eps^2 \int_{\R^3_+}  \eta' \,     (u_a+ u)_3 |\na \fp|^2 +  \eps^2 \int_{\R^3_+}  \eta' \,    \fp   \na (u_a+ u)_3 \cdot \na \fp
+  \eps^2 \int_{\R^3_+}  \eta'' \,    \fp   (u_a+ u)_3 \na \fp.
\end{aligned}
\end{equation}
The first term has a ``good sign'', and will be used to absorb the first term at the r.h.s. of \eqref{estimJ2A}.  The second and third terms can be bounded using properties of the approximate solution and of the weight $\eta$  (notably \eqref{eta''}). 
We end up with 
\begin{equation} \label{estimaK1A}
 K_1 \: \le \:  \eps^2 \int_{\R^3_+}  \eta' \,     (u_a+ u)_3 |\na \fp|^2 \: + \:  C(C_{a}, M) \Big( \mu \| \sqrt{\eta'} (\fp, \eps \nabla \fp) \|_{L^2}^2  +   \| (\fp, \eps \nabla \fp) \|_{L^2}^2\Big). 
\end{equation}
$K_2$ is also a good term, that will absorb the third term at the r.h.s. of \eqref{estimJ2A}. For $K_3$, we use as usual \eqref{hypsource} and write  
$$ K_3 \: \le \: C(C_a, M) \bigl( \| R^r \|_{L^2}^2  \: + \: \mu  \| \sqrt{\eta'} \,  R^s \|_{L^2}^2 \: + \: \mu \| \sqrt{\eta'} \fp \|_{L^2}^2 \: + \:   \| \fp \|_{L^2}^2\bigr). $$
Finally, these last bounds yield
\begin{equation} \label{estimJ3A}
\begin{aligned}
J_3 \:  \le \: & \eps^2 \int_{\R^3_+}  \eta' \,     (u_a+ u)_3 |\na \fp|^2 \: + \:  \int_{\R^3_+}  \eta' \,      (u_a+ u)_3   e^{-\phi_{a}}( 1+ h(\phi)) |\fp|^2 \\
& C(C_a, M) \bigl( \| R^r \|_{L^2}^2  \: + \: \mu  \| \sqrt{\eta'} \,  R^s \|_{L^2}^2 \: + \: \mu \| \sqrt{\eta'}  (\fp, \e \na \fp)  \|_{L^2}^2 \: + \:  \| (\fp, \e \na \fp) \|_{L^2}^2 \bigr). 
\end{aligned}
\end{equation}

\medskip
{\bf Conclusion.} We can now collect the estimates \eqref{estimJ1A}, \eqref{estimJ2A} and \eqref{estimJ3A}, and insert them into \eqref{estimI1A}, followed by \eqref{estimcalIA}. We get 
\begin{equation*}
\begin{aligned}
 & \I \: + \:  \eps^2   \frac{d}{dt}  \int_{\R^3_+}  \frac{1}{2}\eta \,  |\na \fp|^2 \: + \:  \frac{d}{dt}  \int_{\R^3_+} \frac{1}{2}\eta \,  \fp^2 e^{-\phi_{a}} (1 + h(\phi)) \\
& \: \le \: - \:  \int_{\R^3_{+}}  \eta' \, \fp \,  \big( (n_{a}+ n)  \up_3 \big) 
+ \frac{\eps^2}{2} \int_{\R^3_+}  \eta' \,  |\na \fp|^2   (u_a+u)_3 \: + \:  \frac{\eps^2}{2} \int_{x_3=0}  \, |\pa_3 \fp|^2 (u_a + u)_3  \\
&   \: + \: \frac{1}{2}  \int_{\R^3_+}  \eta'  \, (u_a + u)_3 \, e^{-\phi_{a}}(1+ h(\phi))  \,  |\fp|^2\:  
\: + \: \int_{\R^3_+} \eta' \,  |\eps \, \pa_3 \fp|^2 \,   (u^a(t,x) + u)_{3} \: \\
 & \: + \: C(C_a, M) \bigl( \| R^r \|_{L^2}^2  \: + \: \mu  \| \sqrt{\eta'}  R^s \|_{L^2}^2  \: + \: \mu \| \sqrt{\eta'}  (\up, \np, \fp, \e \na \fp)  \|_{L^2}^2  \: + \:  \| (\up, \np, \fp, \e \na \fp) \|_{L^2}^2 \\
 &  \: + \: \| \sqrt{\eta'}\fp \|_{L^2}  \| \sqrt{\eta'} (\eps \, \na \np, \eps \, \div \up) \|_{L^2}.
 \end{aligned}
 \end{equation*}

 By combining this last estimate with the estimate of Lemma \ref{lem+} and using the substitution \eqref{subtrace} in the terms involving $\eta'$, we obtain
the estimate
\begin{equation}
\label{Afin1}
\begin{aligned}
 \frac{d}{dt} \Big[ \int_{\R^3_+}  \eta \, (n_a + n) \frac{|\up|^2}{2} + \int_{\R^3_+}  \eta  \frac{ T^{i} \np^2}{2 (n_a + n) } +  \frac{1}{2}\int_{\R^3_{+}}  \eta \, \eps^2 |\nabla \fp|^2 +  \eta \, e^{\phi_{a}} (1 + h(\phi)) \fp^2\Big] \\
  +  \frac{1}{2} \int_{\R^3_+}  \eta' \,     |\Gamma(u_a+ u)_3| |\eps \, \na \fp|^2  + \frac{1}{2} \int_{\R^3_+}  \eta' \,     |\Gamma(u_a+ u)_3| |\eps \, \pa_3 \fp|^2 + \frac{1}{2} \int_{\R^3_+}  \eta' \, Q^A(\np, \up,\fp) \\
  + \frac{1}{2} \int_{x_3=0} |(u_a + u)_{3}| |\eps \, \pa_3 \fp|^2  + \frac{1}{2} \int_{x_3=0} Q^0 (\np,\up)\\
  \leq  \| \sqrt{\eta'} \fp \|_{L^2} \| \sqrt{\eta'}( \eps \, \nabla \np, \eps \, \div \up) \|_{L^2} +    C(C_{a}, M) \big(
 \| (\np, \up, \fp, \eps \, \na \fp)\|_{L^2}^2  \\
 +  \mu  \| \sqrt{\eta'} \, \eps \nabla \fp \|_{L^2}^2 +  \|R^r \|_{L^2(\mathbb{R}^3_{+})}^2 +  \mu \| \sqrt{\eta'} R^s \|_{L^2(\R^3_{+})}^2  \big),
\end{aligned}
\end{equation}
where $Q^A$ is the positive quadratic form determined by the following symmetric matrix:
$$
M^A =  \Gamma \begin{bmatrix} T^i \frac{|(u_{a} + u)_3|}{n_a + n}& -T^i e^\top& 0 \\ - T^i e& (n_a + n) |(u_{a} + u)_{3} |{Id_{3}}& (n_a + n) e  \\ 0 & (n_a + n) e^\top & e^{- \phi_{a}}(1+ h(\phi)) |(u_{a} + u)_{3}| \end{bmatrix} -  \mu  C I_5.
$$
Since $\mu$ can be chosen as small as we want, it suffices to prove that the leading above matrix is positive. According to Sylvester's criterion, we only have to check that the leading principal minors, denoted by $(\D_i)_{ 1 \leq i \leq 5}$ are positive. We compute:
$$
\D_1 = T^i \Gamma \Big( \frac{|(u_{a} + u)_{3}|}{n_a + n} \Big)>0, \quad  \D_{2}=  T^i  \Gamma\Big(\frac{|(u_{a} + u)_{3}|^2}{n_a + n} \Big) >0, \quad \D_{3}=  T^i  \Gamma \Big(\frac{|(u_{a} + u)_{3}|^3}{n_a + n}  \Big)>0 
$$
$$
\D_4 =  T^i \Gamma  \Big( (n_{a}+ n)^2 | ( u_{a} + u)_{3}|^2 \left( (|u_{a} + u)_{3}|^2-T^i\right) \Big),
$$
which is also positive as a straightforward consequence of Bohm condition \eqref{hypbohm}.
Finally the determinant is equal to:
$$
\D_5= T^i  \Gamma\Big( (|(u_{a} + u)_{3}|^3 (n_a + n)^2 \left( {e^{- \phi_{a} } (1 + h(\phi)) \over n_{a} + n } |(u_{a} + u)_{3}|^2 - (T^i+1)\right) \Big),
$$
which is positive as soon as $|(u_{a} + u)_{3}|>\sqrt{T^i+1}$ as ensured by the Bohm condition \eqref{hypbohm}. Indeed,  
we have that  $|(u, n, \phi)|= O(\eps)$ thanks to \eqref{hypmin}, and that since $n^0= e^{\phi^0}$ that
 $n_{a} = e^{- \phi_{a}} + O(\delta + \eps)$, therefore  
 $$  \Gamma \left({e^{- \phi_{a} } (1 + h(\phi)) \over n_{a} + n } |(u_{a} + u)_{3}|^2 - (T^i+1) \right)=  \Gamma \big(  |(u_{a} + u)_{3}|^2 - (T^i+1) \big) + O(\eps + \delta)>0,$$
  for $\eps $ and $\delta $ sufficiently small. The quadratic form $Q^0$ is also positive as already observed.

To derive \eqref{Afinal} from \eqref{Afin1}, it remains to note that, by positivity  of   $Q^A$,  we can write the Young inequality 
$$\| \sqrt{\eta'} \fp \|_{L^2} \| \sqrt{\eta'}( \eps \, \nabla \np, \eps \, \div \up) \|_{L^2}  \leq  \tilde\mu  \| \sqrt{\eta'} \fp \|_{L^2}^2+  C_{\tilde \mu}
   \| \sqrt{\eta'}( \eps \, \nabla \np, \eps \, \div \up) \|_{L^2}^2,$$
    and absorb the term  $\tilde\mu  \| \sqrt{\eta'} \fp \|_{L^2}^2 $ in the left hand side by choosing $\tilde \mu$ small enough. In the same spirit, 
     the term $ \mu \| \sqrt{\eta'}\,  \eps \nabla \fp \|_{L^2}^2$ can be absorbed by the quadratic terms  in $\eps \na \fp$ at the left hand side of \eqref{Afin1} for  $\mu$  small enough.


\medskip


To conclude this section, we still have to provide the proof of Lemma \ref{lemmafp}. 

\medskip
{\bf Proof of Lemma  \ref{lemmafp}.} To estimate the term  $\eps^2 \int \eta'  \fp \,  \pa_t \pa_{x_3} \fp$ in \eqref{I21bis}, we shall rely on an elliptic estimate on the Poisson equation.
We first have the straightforward bound:
\begin{equation}
\label{easybound}
\begin{aligned}
\left| \int \eps \eta'  \fp \,  \eps \pa_t \pa_{x_3} \fp\right| &\leq \|\sqrt{\eta'} \fp \|_{L^2}\| \sqrt{ \, \eta'} \,  \eps^2 \pa_t \pa_{x_3} \fp  \|_{L^2}.
\end{aligned}
\end{equation}
Differentiating the Poisson with respect to time and taking the product with $\eta' \, \pa_t \fp$, we obtain:
\begin{multline*}
\int_{\R^3_+} \eta'  \, e^{-\phi_a}(1+h(\phi))  |\pa_t \fp|^2 +  \eps^2 \int_{\R^3_+} \eta' \, |\na \pa_t \fp|^2 + \eps^2 \int_{\R^3_+} \eta'' \,  \pa_t \pa_{x_3} \fp  \pa_t \fp \\
= - \int_{\R^3_+} \eta' \, \pa_t \np \, \pa_t \fp - \int_{\R^3_+} \eta' \,  \pa_t(e^{-\phi_a}(1+h(\phi))) \fp \, \pa_t \fp - \int_{\R^3_{+}} \eta' \partial_{t} r_{\phi} \partial_{t} \fp. 
\end{multline*}

We have chosen the parameters of the weight $\eta$ so that  $\delta  /\mu^2 $  is  small enough. Hence, 
$$
 \left|\eps^2 \int_{\R^3_+} \eta'' \,  \pa_t \pa_{x_3} \fp  \,\pa_t \fp\right| \leq \frac{1}{2}  \eps^2 \int_{\R^3_+} \eta' \, |\na \pa_t \fp|^2  + \frac{1}{4}\int_{\R^3_+} \eta'  \, e^{-\phi_a}(1+h(\phi))  |\pa_t \fp|^2.
$$

We also have  by using \eqref{hypmin} and the Young inequality that:
\begin{multline*}
\left|\int_{\R^3_+} - \eta' \, \pa_t \np \, \pa_t \fp - \int_{\R^3_+} \eta' \,  \pa_t(e^{-\phi_a}(1+h(\phi))) \fp \, \pa_t \fp\right| \\
\leq \frac{1}{8} \int_{\R^3_+} \eta'  \, e^{-\phi_a}(1+h(\phi))  |\pa_t \fp|^2 + C(C_{a}, M ) \left( \| \sqrt{\eta'} \pa_t \np\|_{L^2}^2 +   \| \sqrt{\eta'} \fp\|_{L^2}^2\right)
\end{multline*} 
and by using \eqref{hypsource} that
$$  \int_{\R^3_{+}} \eta' \partial_{t} r_{\phi} \partial_{t} \fp \leq C(C_{a}, M) \big( { 1 \over \eps } \| R^r\|_{L^2}^2 + {\mu^2 \over \eps^2}  \| \sqrt{\eta'} \, R^s \|_{L^2}^2\big) + 
 { 1 \over 8}  \int_{\R^3_+} \eta'  \, e^{-\phi_a}(1+h(\phi))  |\pa_t \fp|^2.$$
We thus  end up with the elliptic estimate:
\begin{multline}
\label{3.37}
\int_{\R^3_+} \eta'  \, e^{-\phi_a}(1+h(\phi))  |\pa_t \fp|^2 +    \int_{\R^3_+} \eta' \, | \eps\na \pa_t \fp|^2 
\\
\leq C(C_{a}, M) \left( \| \sqrt{\eta'} \pa_t \np\|_{L^2}^2 +   \| \sqrt{\eta'} \fp\|_{L^2}^2  { 1 \over \eps } \| R^r\|_{L^2}^2 + {\mu^2 \over \eps^2}  \| \sqrt{\eta'} \, R^s \|_{L^2}^2 \right).
\end{multline} 
Going back to \eqref{easybound}, this yields
$$\left| \int \eps \eta'  \fp \,  \eps \pa_t \pa_{x_3} \fp\right| \leq C(C_{a}, M) \|\sqrt{\eta'} \fp \|_{L^2}\big(  \| \sqrt{\eta'} \eps \partial_{t} \np \|_{L^2} + \eps \|  \sqrt{\eta'} \fp \|_{L^2}
 +  \sqrt{\eps} \| R^r\|_{L^2} + \mu    \| \sqrt{\eta'} \, R^s \|_{L^2}   \big).$$
Now using the transport equation  satisfied by $\np$, we can estimate $\| \sqrt{\eta'}  \eps \pa_t  \np\|_{L^2}$. This yields:
$$\| \sqrt{\eta'} \eps  \pa_t \np\|_{L^2}
\leq  C_{a} \| \sqrt{\eta'} \, (\eps \nabla \np, \eps \div \up)\|_{L^2}  +  \sqrt{\eps} \| R^r \|_{L^2} + \mu \| \sqrt{\eta'} R^s \|_{L^2}.
$$
For the last term, we have used that $\eps \eta' \le \mu$ with the choice of the parameters.  We have thus  proven that 
\begin{multline*}
\left| \int \eps \eta'  \fp \,  \eps \pa_t \pa_{x_3} \fp\right| \: \leq \:   \| \sqrt{\eta'}\fp \|_{L^2}  \| \sqrt{\eta'} (\eps \, \na \np, \eps \, \div \up) \|_{L^2}
   \\+  C(C_{a}, M) \big( \|R^r\|_{L^2}^2 + \mu \| \sqrt{\eta'} R^s\|_{L^2}^2 +  \mu \| \sqrt{\eta'} \fp \|_{L^2}^2 \: + \:   \|\fp  \|^2_{L^2} \big).
\end{multline*}

\bigskip
{\bf B. The second energy estimate}

The previous computation suggests to look for a control on $\eps \, \na \np$ and $\eps \, \div \up$, and this motivates the next energy estimate.  We shall prove that 
\begin{multline}
\label{Bfinal}
 \|(\np, \up, \eps\, \nabla \np,  \eps\,  \div \up )\|_{L^\infty_{T}L^2(\mathbb{R}^3_{+})}^2 +   \| \sqrt{\eta'}(\np, \up,  \eps\, \nabla \np, \eps\, \div \up) \|_{L^2_{T} L^2(\mathbb{R}^3_{+})}^2  \\ 
  + \|\Gamma( \np,  \up, \eps \partial_{x_{3}} \np, \eps\, \div \up) \|_{L^2_{T}L^2(\mathbb{R}^2)} \\
   \leq C(C_{a}, M) \Big( \|(\np, \up, \eps \nabla \np, \eps \div \up)(0)\|_{L^2(\mathbb{R}^3_{+})}^2 + \mu   \|  \sqrt{\eta'}\,  \eps \nabla \up \|_{L^2_{T}L^2(\mathbb{R}^3_{+})}^2
    \\+ \ \|(\np, \up, \eps \nabla \np, \eps \nabla \up , \fp, \eps \nabla \fp)\|_{L^2_{T}L^2(\mathbb{R}^3_{+})}^2
     + \|R^r \|_{L^2_{T}L^2(\mathbb{R}^3_{+})}^2 + \mu \| \sqrt{\eta'} R^s \|_{L^2_{T}L^2(\R^3_{+})}^2
    \Big).
\end{multline} 
Let us stress that this estimate alone does not allow to conclude : the right hand side involves the full $\eps \nabla \up$, while the left hand side controls only  $\eps \div \up$. Later, we shall establish a third estimate,  in order to handle  $\eps \nabla \up$ and close the argument.

To prove \eqref{Bfinal}, we shall again rely on  Lemma \ref{lem+}, but $\I$ will be handled in a different way. By integration by parts:
\begin{align}
\nonumber
\I & = \int_{\R^3_+} \eta  \na \fp \, (n_a + n) \cdot  \up \\ 
\nonumber
   & = -\int_{\R^3_+} \eta   \fp \, (n_a + n) \div \up  \: - \:  \int_{\R^3_+} \eta'   \fp \, (n_a + n)  \up_3 \:  - \:   \int_{\R^3_+} \eta   \fp \, \na (n_a + n) \cdot \up   \\
  \label{estimcalIB} 
& \le \:  I_1  \: + \: I_2 \: +  \: C(C_a,M) \left( \mu \| \sqrt{\eta'} (\up,\fp) \|_{L^2}^2 \: + \:    \| (\up,\fp) \|_{L^2}^2 \right) 
\end{align}
with 
$$I_1 \:  := \:  -\int_{\R^3_+} \eta   \fp \, (n_a + n) \div \up, \quad I_2 \: := \: - \:  \int_{\R^3_+} \eta'   \fp \, (n_a + n)  \up_3.  $$
The bilinear singular term $I_2$ will be ``compensated'' in the end (thanks to the Bohm condition). Therefore we first focus on $ I_1$. 

\medskip
{\bf 1) Treatment of $I_1$.} A first idea is to use the Poisson equation satisfied by $\fp$: 
\begin{equation}
\fp = \frac{1}{1+h(\phi)} e^{\phi_a} \left[ \eps^2 \Delta \fp - \np - r_\phi\right],
\end{equation}
to express $\fp$ in terms of $\eps^2 \Delta \fp$ and $\np$. We get:
\begin{align*}
I_1 & = - \int_{\R^3_+} \eta (n_a + n)\frac{1}{1+h(\phi)} e^{\phi_a} \div \up  \left[ \eps^2 \Delta \fp - \np - r_\phi\right] \\
& \le   J_1 \: + \: J_2 \: + C \int_{\R^3_+}  | \eps \div \up | \, \left| \frac{r_\phi}{\eps}\right|  
\end{align*}
where 
$$ J_1 \: :=  \: - \int_{\R^3_+} \eta (n_a + n)\frac{1}{1+h(\phi)} e^{\phi_a} \div \up  \, (\eps^2 \Delta \fp), \quad J_2 :=   \int_{\R^3_+} \eta (n_a + n)\frac{1}{1+h(\phi)} e^{\phi_a} \div \up \, \np. $$
By \eqref{hypsource}, 
\begin{equation} \label{estimI1B}
I_1  \: \le \: J_1 + J_2 \: + \: C(C_a,M) \, \big( \| R^r \|_{L^2}^2 + \mu \| \sqrt{\eta'} R^s \|_{L^2}^2 \: + \: \mu \| \sqrt{\eta'} \eps \div \up \|_{L^2}^2 \: + \:  \| \eps \div \up \|_{L^2}^2 \big). 
\end{equation}
 We shall now study the terms $J_1,J_2$.

\smallskip
{\bf a. Study of $J_1$.} To evaluate $J_1$, the idea is to take the divergence in the equation satisfied by $\up$, in order to express $\Delta \fp$ in terms of $\up$ and $\np$. This reads:
\begin{equation}
\label{3.48}
(\pa_t + (u_a +u) \cdot \nabla ) (\eps \, \div \up) + \sum_{i=1}^3 \eps \pa_i(u_a +u) \cdot \na \up_i  
 + \, \eps T^i \div  \left(\frac{\na \np}{n_a + n} \right) = \eps \Delta \fp + \eps \div r_u.
\end{equation}
Hence,  $J_1 = \sum_{i=1}^4 K_i$, with 
\begin{align*}
 K_1 & := - \int_{\R^3_+} \eta \frac{n_a + n}{1 + h(\phi)} e^{\phi_a} \left( \pa_t + (u_a + u) \cdot \na \right) \frac{|\eps \div \up|^2}{2}, \\
K_2  & := -  \int_{\R^3_+} \eta \frac{n_a + n}{1 + h(\phi)} e^{\phi_a} (\eps \pa_i u_a \cdot \na) \up_i \, (\eps \div \up), \\
K_3 &  := - T^i \int_{\R^3_+} \eta \frac{n_a + n}{1 + h(\phi)} e^{\phi_a} \eps  \div  \left(\frac{\na \np}{n_a + n} \right)  (\eps \div \up), \\
K_4 &  :=  \int_{\R^3_+} \eta \frac{n_a + n}{1 + h(\phi)} e^{\phi_a}   (\eps \div r_u) \,  (\eps \div \up). 
\end{align*}
The terms $K_2$ and $K_4$ can be bounded through standard arguments. With \eqref{betagrand} and \eqref{hypsource} in mind, we obtain
\begin{equation} \label{estimK2K4B}
K_2 + K_4 \: \le \: C(C_a, M) \, \big(  \| R^r \|_{L^2}^2 + \mu \| \sqrt{\eta'} R^s \|_{L^2}^2 \: + \: \mu \| \sqrt{\eta'} \eps \na \up \|_{L^2}^2 \: + \:  \| \eps \na \up \|_{L^2}^2 \big). 
\end{equation}
Note the presence of $\eps \na \up$ at the r.h.s., due to $K_2$.
We now evaluate the contribution of the convection term $K_1$. Standard manipulations yield
\begin{align*}
K_1 = & -\frac{d}{dt}  \int_{\R^3_+}\eta \frac{n_a +n}{1 + h(\phi)} e^{\phi_a} \frac{(\eps \, \div \up)^2}{2} +   \int_{\R^3_+}\eta \frac{d}{dt}\left[\frac{n_a +n}{1 + h(\phi)} e^{\phi_a} \right] \frac{(\eps \, \div \up)^2}{2} \\
& + \int_{\R^3_+} \eta \, \div\left(\frac{n_a +n}{1 + h(\phi)} e^{\phi_a} (u_a + u)\right)  \frac{(\eps \, \div \up)^2}{2} \\
& + \int_{\R^3_+} \eta' \frac{n_a +n}{1 + h(\phi)} e^{\phi_a}(u_a + u)_3  \frac{(\eps \, \div \up)^2}{2} \:  + \: \int_{x_3=0}  \left[\frac{n_a + n}{1 + h(\phi)} e^{\phi_a} (u_a + u)_3 \frac{(\eps \, \div \up)^2}{2}\right].
\end{align*}
The second term can be bounded by a crude $L^2$ estimate, the third one can be bounded thanks to \eqref{boundna0}-\eqref{betagrand}: we get
\begin{equation} \label{estimK1B}
\begin{aligned}
K_1 \:  \le \: & -\frac{d}{dt}  \int_{\R^3_+}\eta \frac{n_a +n}{1 + h(\phi)} e^{\phi_a} \frac{(\eps \, \div \up)^2}{2} \: + \: \int_{\R^3_+} \eta' \frac{n_a +n}{1 + h(\phi)} e^{\phi_a}(u_a + u)_3  \frac{(\eps \, \div \up)^2}{2} \\
& + \: \int_{x_3=0}  \left[\frac{n_a + n}{1 + h(\phi)} e^{\phi_a} (u_a + u)_3 \frac{(\eps \, \div \up)^2}{2}\right] \: + \: C(C_a,M)  \big( \mu \| \sqrt{\eta'} \eps \div \up \|_{L^2}^2 \: + \:  \| \eps \div \up \|_{L^2}^2 \big).
\end{aligned}
\end{equation} 
The  second term is non-positive, and will be one of the crucial terms to control all singular terms (thanks to the Bohm condition).
Note finally that the boundary contribution is non-positive: this will also help later.

Now we turn to the integral that involves the pressure  term, for which we perform an integration by parts
\begin{align*}
K_3 &=  T^i \int_{\R^3_+} \eta  \left(\frac{1}{n_a +n} \frac{1}{1+h(\phi)} e^{\phi_a}\right)\na[(n_a +n)(\eps \, \div \up)] \cdot \eps {\na \np} \\
&+ T^i \int_{\R^3_+} \eta \, \na \left(\frac{1}{n_a +n} \frac{1}{1+h(\phi)} e^{\phi_a}\right)(n_a +n)(\eps \, \div \up) \cdot  \eps {\na \np}\\
&+ {T^i \int_{\R^3_+} \eta' \left(\frac{1}{1+h(\phi)} e^{\phi_a}\right)(\eps \, \div \up) \,  \eps {\pa_3 \np} }\: + \:  T^i \int_{x_3=0}  \left(\frac{1}{1+h(\phi)} e^{\phi_a}\right)(\eps \, \div \up) \,  \eps {\pa_3 \np} \\
&=: L_1 \: + \: L_2 \: + \: L_3 \: + \: L_4.
\end{align*}
We use the equation satisfied by $\np$ in order to rewrite the $\na [...]$ part in $L_1$. To this end, we take the $\eps \na$ operator on the transport equation for $n$ to get: 
\begin{equation}
(\pa_t  + (u_a +u)\cdot \na )(\eps \na \np) + \na \left[(n_a +n) \, \eps \div \up\right] + \eps\sum_{i=1}^3 (\pa_i(u_a +u) \cdot \na \np) e_i  = \eps \na r_n,
\end{equation}
where $e_1 =(1,0,0)^\top, e_2=(0,1,0)^\top, e_3=(0,0,1)^\top$.
As before the convection term on $(\eps \na \np)$ is very useful. With the usual manipulations, we obtain:
\begin{align*}
L_1 \: \leq \: & - \frac{d}{dt} T^i \int_{\R^3_+} \eta  \left(\frac{1}{n_a +n} \frac{1}{1+h(\phi)} e^{\phi_a}\right) \frac{|\eps {\na \np}|^2}{2}\\
& +T^i \int_{\R^3_+} \eta'  \left(\frac{1}{n_a +n} \frac{1}{1+h(\phi)} e^{\phi_a}\right)  (u_a +u)_3 \frac{|\eps \na \np|^2}{2}  \\
& +  T^i\int_{x_3=0}  \left(\frac{1}{n_a +n} \frac{1}{1+h(\phi)} e^{\phi_a}\right)  (u_a +u)_3 \frac{|\eps \na \np|^2}{2} \\
& + \: C(C_a,M) \big( \| R^r \|_{L^2}^2 + \mu \| \sqrt{\eta'} R^s \|_{L^2}^2 \: + \: \mu \| \sqrt{\eta'} \eps \na \np \|_{L^2}^2 \: + \:  \| \eps \na \np \|_{L^2}^2 \big).
\end{align*} 
By now standard manipulations, we also have 
$$ L_2 \: \le \:  C(C_a,M) \big(  \: \mu \| \sqrt{\eta'} (\eps \na \np, \eps \div \up) \|_{L^2}^2 \: + \:  \| (\eps \na \np, \eps \div \up) \|_{L^2}^2 \big). $$
From the estimates on the $L_i$'s, one can deduce an inequality on $K_3$.  Combining this inequality with \eqref{estimK2K4B} and \eqref{estimK1B}, we end up with 
\begin{equation} \label{estimJ1B}
\begin{aligned}
J_1 \le & -\frac{d}{dt}  \int_{\R^3_+}\eta \frac{n_a +n}{1 + h(\phi)} e^{\phi_a} \frac{(\eps \, \div \up)^2}{2} \:  - \:  \frac{d}{dt} T^i \int_{\R^3_+} \eta  \left(\frac{1}{n_a +n} \frac{1}{1+h(\phi)} e^{\phi_a}\right) \frac{|\eps {\na \np}|^2}{2}\\ 
& +  \int_{\R^3_+} \eta' \frac{n_a +n}{1 + h(\phi)} e^{\phi_a}(u_a + u)_3  \frac{(\eps \, \div \up)^2}{2} \: + \: \int_{x_3=0}  \left[\frac{n_a + n}{1 + h(\phi)} e^{\phi_a} (u_a + u)_3 \frac{(\eps \, \div \up)^2}{2}\right] \\
& +T^i \int_{\R^3_+} \eta'  \left(\frac{1}{n_a +n} \frac{1}{1+h(\phi)} e^{\phi_a}\right)  (u_a +u)_3 \frac{|\eps \na \np|^2}{2} \: +\:  T^i \int_{\R^3_+} \eta' \left(\frac{1}{1+h(\phi)} e^{\phi_a}\right)(\eps \, \div \up) \,  \eps {\pa_3 \np}  \\
& +  T^i\int_{x_3=0}  \left(\frac{1}{n_a +n} \frac{1}{1+h(\phi)} e^{\phi_a}\right)  (u_a +u)_3 \frac{|\eps \na \np|^2}{2} \: + \: 
   T^i \int_{x_3=0}  \left(\frac{1}{1+h(\phi)} e^{\phi_a}\right)(\eps \, \div \up) \,  \eps {\pa_3 \np} \\
&   +  C(C_a,M) \big( \| R^r \|_{L^2}^2 + \mu \| \sqrt{\eta'} R^s \|_{L^2}^2 \: + \: \mu \| \sqrt{\eta'} (\eps \na \np, \eps \na \up) \|_{L^2}^2 \: + \:  \| (\eps \na \np, \eps \na \up) \|_{L^2}^2).
\end{aligned}
\end{equation}

{\bf b. Study of $J_2$}. Let us now work on $J_2$, which has to be treated very carefully. Indeed, a rough $L^2$ estimate only shows that this is a singular term in $1/\eps$, which does not seem small. Instead, we use the transport equation satisfied by $\np$ to express $(n_a+ n) \div \up$ in terms of other quantities. With similar manipulations as before, we obtain:
\begin{equation}
\label{estimJ2B}
\begin{aligned}
J_2 &\leq -\frac{d}{dt} \int_{\R^3_+} \eta \frac{1}{1 + h(\phi)} e^{\phi_a} \frac{|\np|^2}{2} \\
  & + C(C_{a}, M)\big( \|R^r \|_{L^2}^2 + \mu \| \sqrt{\eta'} R^s \|_{L^2}^2  + \mu \| \sqrt{\eta'}  (\np, \up) \|_{L^2}^2 \: + \:  \| (\np, \up) \|_{L^2}^2  \big)\\
&+ \frac{1}{2} \int_{\R^3_+} \eta' \frac{1}{1 + h(\phi)}e^{\phi_a} (u_a + u)_3 |\np|^2 + \frac{1}{2}  \int_{x_3=0} \frac{1}{1 + h(\phi) }e^{\phi_a} (u_a + u)_3 |\np|^2
\end{aligned}
\end{equation}
Observe here that the last two terms are non-positive.  Together with \eqref{estimJ1B} and \eqref{estimJ2B}, the bound \eqref{estimI1B} leads to 
\begin{equation} \label{estimI1Bbis}
\begin{aligned}
I_1 \le & -\frac{d}{dt}  \int_{\R^3_+}\eta \frac{n_a +n}{1 + h(\phi)} e^{\phi_a} \frac{(\eps \, \div \up)^2}{2} \:  - \:  \frac{d}{dt} T^i \int_{\R^3_+} \eta  \left(\frac{1}{n_a +n} \frac{1}{1+h(\phi)} e^{\phi_a}\right) \frac{|\eps {\na \np}|^2}{2} \\ 
& -\frac{d}{dt} \int_{\R^3_+} \eta \frac{1}{1 + h(\phi)} e^{\phi_a} \frac{|\np|^2}{2}  \\
&  + \frac{1}{2} \int_{\R^3_+} \eta' \frac{1}{1 + h(\phi)}e^{\phi_a} (u_a + u)_3 |\np|^2 + \frac{1}{2}  \int_{x_3=0} \frac{1}{1 + h(\phi) }e^{\phi_a} (u_a + u)_3 |\np|^2 \\
& +  \int_{\R^3_+} \eta' \frac{n_a +n}{1 + h(\phi)} e^{\phi_a}(u_a + u)_3  \frac{(\eps \, \div \up)^2}{2} \: + \: \int_{x_3=0}  \left[\frac{n_a + n}{1 + h(\phi)} e^{\phi_a} (u_a + u)_3 \frac{(\eps \, \div \up)^2}{2}\right] \\
& +T^i \int_{\R^3_+} \eta'  \left(\frac{1}{n_a +n} \frac{1}{1+h(\phi)} e^{\phi_a}\right)  (u_a +u)_3 \frac{|\eps \na \np|^2}{2} \: +\:  T^i \int_{\R^3_+} \eta' \left(\frac{1}{1+h(\phi)} e^{\phi_a}\right)(\eps \, \div \up) \,  \eps {\pa_3 \np}  \\
& +  T^i\int_{x_3=0}  \left(\frac{1}{n_a +n} \frac{1}{1+h(\phi)} e^{\phi_a}\right)  (u_a +u)_3 \frac{|\eps \na \np|^2}{2} \: + \: 
   T^i \int_{x_3=0}  \left(\frac{1}{1+h(\phi)} e^{\phi_a}\right)(\eps \, \div \up) \,  \eps {\pa_3 \np} \\
&   +  C(C_a,M) \big( \| R^r \|_{L^2}^2 + \mu \| \sqrt{\eta'} R^s \|_{L^2}^2 \: + \: \mu \| \sqrt{\eta'} (\np, \up,\eps \na \np, \eps \na \up) \|_{L^2}^2 \: + \:  \| (\np, \up,\eps \na \np, \eps \na \up) \|_{L^2}^2).
\end{aligned}
\end{equation}

\medskip
{\bf 2) Treatment of $I_2$}

By using \eqref{subtrace}, we have the straightforward bound:
\begin{equation}
\begin{aligned}
I_2  \: \leq \: &   \frac{1}{2} \int_{\R^3_+} \eta'    \, (n_a + n)  |(u_a+u)_3| |\fp|^2  + \frac{1}{2}  \int_{\R^3_+} \eta'  \, \frac{n_a + n}{|(u_a+u)_3|} |\up_3|^2 \\
  \leq \: &  \frac{1}{2} \int_{\R^3_+} \eta'    \,  \Gamma n^0   |\Gamma (u^0)_3| |\fp|^2  + \frac{1}{2}  \int_{\R^3_+} \eta'  \, \frac{n_a + n}{|(u_a+u)_3|} |\up_3|^2  \: + \:  \mu \int_{\mathbb{R}^3_{+}} \eta' |\fp|^2 \, dx \\
 &   +  C(C_{a}, M) \big( 1 + {\|(u,n)\|_{L^\infty} \over \eps} \big) \|\fp \|_{L^2}^2.
\end{aligned}
\end{equation}
Note that the last term is actually not singular, thanks to \eqref{hypmin}. 

\medskip
We now claim that $\fp$ satisfies the following bound:
\begin{multline} \label{ellipticB}
 \frac{1}{2} \int_{\R^3_+} \eta'    \,  \Gamma n^0   |\Gamma (u^0)_3| |\fp|^2  \leq  \frac{1}{2}  \int_{\R^3_+} \eta'    \,  \frac{|(u_a+u)_3|}{n_a+n} |\np|^2
    \\
  +  C(C_{a}, M)\Big(  \left\|  R^r \right \|_{L^2}^2  + \mu \| \sqrt{\eta'} R^s \|_{L^2}^2 \: + \: \mu \| \sqrt{\eta'} (\np, \fp) \|_{L^2}^2 \: + \: 
   \|  (\np, \fp) \|_{L^2}^2 \Big).
\end{multline}
 The term   $\frac{1}{2}  \int_{\R^3_+} \eta'    \,  \frac{|(u_a+u)_3|}{n_a+n} |\np|^2$
 has the bad sign but will be  compensated by other terms  later. 
 
 To prove  estimate \eqref{ellipticB},  we write the Poisson equation satisfied by $\fp$:
$$ -\eps^2 \Delta \fp + {e^{- \phi_{a}}}( 1  + h(\phi) \fp  = -\np  - r_{\phi}.$$
We multiply by $m \fp$, with the weight $ m \: := \:  \eta'     |(\Gamma u^0)_3|$. By a  standard energy estimate:  
 $$ \| \sqrt{m}\, \eps \nabla \fp \|_{L^2}^2 + \int_{\mathbb{R}^3} m e^{- \phi_{a}}( 1 + h(\phi)) | \fp |^2 \, dx\leq 
  \int_{\mathbb{R}^3} m (|\np| + |r_{\phi}|) |\fp|\, dx +  \eps^2 \int_{\mathbb{R}^3} | \nabla^2 m|   |\fp|^2.$$
We use the Young inequality for the first term, resulting in 
$$   \int_{\mathbb{R}^3} m |\np| \, |\fp| \: \le \:   { 1 \over 2}\int_{\mathbb{R}^3}
   \eta' \left|{\Gamma (u^0)_{3} \over \Gamma n^0 } \right| \, |\np|^2  \: + \:  \frac{1}{2}\int_{\mathbb{R}^3} \eta'   ( \Gamma n^0   |\Gamma (u^0)_3 |) |\fp |^2. $$

Next, we observe that  thanks to \eqref{hypsource}, we have 
$$  \int_{\R^3_{+}} m |r_{\phi}|\, |\fp|  \leq  C(C_{a}, M) \int_{\R^3_{+}} {|r_{\phi}| \over \eps}  | \fp |
 \leq C(C_{a}, M)  \big( \| R^r \|_{L^2}^2 +  \mu \|\sqrt{\eta'} R^s\|_{L^2}^2 \, +  \, \mu \| \sqrt{\eta'} \fp \|_{L^2}^2 + \| \fp \|_{L^2}^2 \big).$$
By the choice of the weight function, we have
$$ \eps^2 \, |\nabla^2 m| \lesssim   \left(\mu^2 + \delta + \eps^2 \right)m  \lesssim \mu^2 m,$$
which implies  
$$\eps^2 \int_{\mathbb{R}^3} | \nabla^2 m|   |\fp|^2 \: \lesssim \: \mu^2 \| \sqrt{m}\fp  \|_{L^2}^2   \: \le \: C_a \: \mu^2  \| \sqrt{\eta'} \fp \|_{L^2}^2. $$
 Finally, since  $n^0 = e^{- \phi_{0}},$  relations \eqref{app0}, \eqref{subtrace} give
\begin{align*}
 \int_{\mathbb{R}^3} m  e^{- \phi_{a}}( 1 + h(\phi)) | \fp |^2 \,  & \geq \int_{\mathbb{R}^3}  \eta'   (\Gamma n^0  |\Gamma (u^0)_3 | -  C_{a} (\delta + \eps))  |\fp|^2 \:  - \:   C_a \int_{\mathbb{R}^3}   ( 1 + {\| \phi \|_{L^\infty} \over \eps})  |\fp|^2 \\
 & \geq \int_{\mathbb{R}^3}  \eta'   (\Gamma n^0  |\Gamma (u^0)_3 ||\fp|^2   \: - \: C(C_a,M) \big( \mu \|  \sqrt{\eta'}  \fp \|_{L^2}^2  \: + \: \| \fp \|_{L^2}^2).  
 \end{align*}
  We thus get  that
  \begin{multline} \label{fpbynp}
  \frac{1}{2}\int_{\mathbb{R}^3} \eta'   ( \Gamma n^0   |\Gamma (u^0)_3 |) |\fp |^2    \leq { 1 \over 2}\int_{\mathbb{R}^3}
   \eta' \left|{\Gamma (u^0)_{3} \over \Gamma n^0 } \right| \, |\np|^2 \,  \\
   + C(C_{a}, M)\Big(  \left\|  R^r \right \|_{L^2}^2  + \mu \| \sqrt{\eta'} R^s \|_{L^2}^2 \: + \: \mu \| \sqrt{\eta'} \fp \|_{L^2}^2 \: + \: 
   \|  \fp \|_{L^2}^2 \Big).
   \end{multline}
   Applying \eqref{subtrace2} to the first term at the r.h.s., we obtain the desired inequality \eqref{ellipticB}.It follows that 
\begin{equation} \label{estimI2B}
\begin{aligned}
I_2 \: \le \: & \frac{1}{2}  \int_{\R^3_+} \eta'    \,  \frac{|(u_a+u)_3|}{n_a+n} |\np|^2  + \frac{1}{2}  \int_{\R^3_+} \eta'  \, \frac{n_a + n}{|(u_a+u)_3|} |\up_3|^2 \\
   & + C(C_{a}, M)\Big(  \left\|  R^r \right \|_{L^2}^2  + \mu \| \sqrt{\eta'} R^s \|_{L^2}^2 \: + \: \mu \| \sqrt{\eta'} \fp \|_{L^2}^2 \: + \: 
   \|  \fp \|_{L^2}^2 \Big).
   \end{aligned}
   \end{equation}

  \bigskip
  {\bf 3) Conclusion of estimate B.}  We can then inject \eqref{estimI1Bbis} and \eqref{estimI2B} into \eqref{estimcalIB}. We get
  \begin{equation}
\begin{aligned}
\I + \frac{d}{dt}  \int_{\R^3_+}\eta \Big( \frac{n_a +n}{1 + h(\phi)} e^{\phi_a}  \frac{(\eps \, \div \up)^2}{2} 
+ \ \left(\frac{T^{i}}{n_a +n} \frac{1}{1+h(\phi)} e^{\phi_a}\right) \frac{|\eps {\na \np}|^2}{2}  
+   \frac{1}{1 + h(\phi)} e^{\phi_a} \frac{|\np|^2}{2}  \Big)\\
\leq \int_{\R^3_+} \eta' \frac{n_a +n}{1 + h(\phi)} e^{\phi_a}  (u_a + u)_3  \frac{(\eps \, \div \up)^2}{2} \\
+  T^i \int_{\R^3_+} \eta'  \left(\frac{1}{n_a +n} \frac{1}{1+h(\phi)} e^{\phi_a}\right)  (u_a +u)_3 \frac{|\eps \na \np|^2}{2}  
+ T^i \int_{\R^3_+} \eta' \left(\frac{1}{1+h(\phi)} e^{\phi_a}\right)(\eps \, \div \up) \,  \eps {\pa_3 \np} 
\\  + { 1 \over 2} \int_{\R^{3}_{+}} \eta'  \Big({  e^{\phi_{a}} \over 1 + h(\phi) }  (u_{a} + u )_{3}\, |\np|^2   +  { |(u_{a} + u)_{3} | \over n_{a} + n} \, |\np |^2 
+  \, \frac{n_a + n}{|(u_a+u)_3|} |\up_3|^2\Big)  \\
 + \mu \| \sqrt{\eta'} \, |\eps \, \nabla \up \|_{L^2}^2 \: + \: \mu \|  \sqrt{\eta'} \, \fp \|_{L^2}^2   + B.T. + \mathcal{G} + C(C_{a}, M ) \mu \| \sqrt{\eta'} (\np, \up,  \eps \nabla \np,  \eps \div \up) \|_{L^2(\R^3_{+})}^2.
\end{aligned}
\end{equation}
Here $B.T.$ denotes the boundary contributions, 
\begin{multline} 
B.T. :=  \int_{x_{3}= 0}  \Big(  {1 \over 2} e^{\phi_{a}} (n_{a} + n)  (u_{a} + u)_{3}  (\eps \div \up )^2 + {T^i \over 2 } { e^{\phi_{a}  }  \over( n_{a} + n)  (1+h(\phi)} ( u_{a} + u)_{3}
 | \eps \nabla \np |^2   \\+ T^i  { e^{\phi_{a}} \over (1 + h(\phi))} \eps \div \up \, \eps \partial_{3} \np + { 1 \over 2} { e^{\phi_{a}} \over 1 + h(\phi)} ( u_{a} + u)_{3} | \np |^2 \Big)
 \end{multline}
 and $\mathcal{G}$ stands for a good term, that is
\begin{multline}
\label{Good}
 \mathcal{G} \leq  C(C_{a}, M ) \Big( \|(\np, \up, \eps \nabla \np, \eps \nabla \up , \fp, \eps \nabla \fp)\|_{L^2(\mathbb{R}^3)}^2 
    + \|  R^r \|_{L^2(\mathbb{R}^3)}^2 + \mu  \|\sqrt{\eta'} R^s \|_{L^2(\mathbb{R}^3)}^2 \Big).
     \end{multline}
Let us stress that some singular terms vanish at leading order : indeed, the relation $e^{\phi_0} = 1/n_0$ implies 
$$ { 1 \over 2} \int_{\R^{3}_{+}} \eta'  \Big({  e^{\phi_{a}} \over 1 + h(\phi) }  (u_{a} + u )_{3}\, |\np|^2   +  { |(u_{a} + u)_{3} | \over n_{a} + n} \, |\np |^2  \: \le \:  C(C_a,M) \, \left( \mu \| \sqrt{\eta'} \np \|_{L^2}^2 + \| \np \|_{L^2}^2 \right). $$

\medskip
By combining  the previous estimate and Lemma \ref{lem+}, we conclude that we have

\begin{equation} \label{almostBfinal}
\begin{aligned}
\frac{d}{dt}  \Big[ \int_{\R^3_+}  \eta \, (n_a + n) \frac{|\up|^2}{2} +  \int_{\R^3_+}  \eta  \frac{ T^{i} |\np|^2}{2 (n_a + n) } + \int_{\R^3_+}\eta \frac{1}{1 + h(\phi)} e^{\phi_a} (n_a +n ) \frac{(\eps \, \div \up)^2}{2} \\
+ T^i \int_{\R^3_+} \eta  \left(\frac{1}{n_a +n} \frac{1}{1+h(\phi)} e^{\phi_a}\right) \frac{|\eps {\na \np}|^2}{2}  
+  \int_{\R^3_+} \eta \frac{1}{1 + h(\phi)} e^{\phi_a} \frac{|\np|^2}{2} \Big] \\
+ \frac{1}{2} \int_{\R^3_+}  \eta' \, Q_1^B(\np,\up) + \frac{1}{2} \int_{\R^3_+}  \eta' \, Q_2^B(\eps \, \div \up, \eps \nabla  \np) 
\\
+\frac{1}{2} \int_{x_3=0} Q^0(\np,\up) + \frac{1}{2}  \int_{x_3=0} Q^B_{2}(\eps \div \up, \eps \nabla \np) \leq \mu \| \sqrt{\eta'} \, |\eps \, \nabla \up \|_{L^2}^2 \: + \:  \mu \|  \sqrt{\eta'} \, \fp \|_{L^2}^2    \: + \:   \mathcal{G}. 
\end{aligned}
\end{equation}

with $Q_1^B, Q_2^B$ are quadratic forms determined by the two following symmetric matrices:
$$
M_1^B = \Gamma  \begin{bmatrix}  T^i {|(u_{a} + u)_{3}|\over n^a+ n}  & -T^i  e^\top\\ - T^i e & (n_a+ n)\left(|(u_{a} + u)_3| Id_{3}-\frac{1}{|(u_{a} + u)_{3}|} e \, e^\top\right)  \end{bmatrix} - C  \mu I_4, $$
$$ M_2^B = \Gamma  \begin{bmatrix} {e^{\phi_{a} }\over 1 + h(\phi)} (n_{a} + n) |(u_{a} + u)|_{3} & -\frac{T^i e^{\phi_{a}} }{1+ h(\phi)}  e^\top  \\ -\frac{T^i e^{\phi_{a}} }{1+ h(\phi)}  e  & T^i\frac{ e^{\phi_a} |(u_{a} + u)_{3}|}{(1+h(\phi))(n_a+ n)} Id_{3} \end{bmatrix}  -  C \mu I_4.
$$
 We recall that $\mu>0$ is a parameter that can be taken arbitrarily small. Observe that $M_1^B$ is positive for $\mu$ sufficiently small, as soon as the Bohm condition \eqref{hypbohm} is satisfied.
   We also get  that  $M_2^B$ is  positive thanks to  the Bohm condition.

To obtain \eqref{Bfinal}, we still have to get rid of the term   $\mu \|  \sqrt{\eta'} \, \fp \|_{L^2}^2$ at the r.h.s. of \eqref{almostBfinal}. But as the quadratic form $Q_1^B$ controls $\| \sqrt{\eta'} \np \|_{L^2}^2$, it also controls  $\| \sqrt{\eta'} \fp \|_{L^2}^2$ thanks to  \eqref{fpbynp}, up to add some good term and some $\mu \| \sqrt{\eta'} \np \|_{L^2}^2$ term at the r.h.s. By taking $\mu$ small enough, this singular term is controlled by the l.h.s., which gives \eqref{Bfinal}.

\smallskip

This estimate by itself allows us almost to conclude; the only problematic term comes from the singular contribution $\int_{\R^3_+} \eta' \, |\eps \, \nabla  \up|^2$ in the right hand side  which involves a full derivative of $\up$, and that will be treated below.

\bigskip

{\bf C. The third energy estimate }

\bigskip

The previous computations suggest that we also need an estimate of a full $\eps$-derivative of $\up$.
Note that we can not work  with the $\curl$ of $\up$ (together with the previous estimate on the divergence), because we would not be able to control  the trace of $\up$ on the boundary $\{x_3=0\}$. 
 
 We shall prove that 
 \begin{multline}
 \label{Cfinal}
  \| (\np, \up, \eps \nabla \np, \eps \nabla \up )\|_{L^\infty_{T} L^2(\mathbb{R}^3_{+})}^2 + \| \sqrt{ \eta '} (\np, \up, \eps \nabla \np, \eps \nabla \up ) \|_{L^2_{T} L^2 (\mathbb{R}^3_{+})}^2 + \| \Gamma (\np, \up, \eps \nabla \np, \eps \nabla \up ) \|_{L^2_{T} L^2 (\mathbb{R}^2)} \\
   \leq  C(C_{a}, M) \Big( \|(\np, \up, \eps \nabla \np, \eps \nabla \up)(0)\|_{L^2(\mathbb{R}^3_{+})}^2 + \big\| \sqrt{\eta'} \big(\np, \up, \fp,  \eps \nabla \np, \eps \div \up, \eps \nabla \fp\big)\|_{L^2_{T}L^2(\mathbb{R}^3_{+})}^2 \\
  +   \| \Gamma ( \np, \up, \eps\partial_{x_{3}} \fp, \eps \nabla \np, \eps \div \up) \|_{ L^2_{T}L^2(\mathbb{R}^2)}^2 
    + 
    \|(\np, \up, \eps \nabla \np, \eps \nabla \up , \fp, \eps \nabla \fp)\|_{L^2_{T}L^2(\mathbb{R}^3)}^2 
    \\+ \| R^r \|_{L^2_{T}L^2(\mathbb{R}^3)}^2 + \mu \| \sqrt{\eta'} R^s \|_{L^2 _{T}L^2(\R^3_{+})}^2 \Big).
 \end{multline}
 
 \medskip
We shall use  again  the estimate of Lemma \ref{lem+}, but through a different approach of the integral $\I$.  The idea will be to combine
 the estimate of Lemma \ref{lem+}  with an estimate on the $\eps-$derivatives of the system to cancel the singular term through the Poisson equation.
 
Let us first  first apply  $\eps$-derivatives on the equation on the velocity field $\up$: for $i=1,2,3$, this yields
\begin{multline}
\label{epsder1}
  \pa_t (\eps \,\pa_i \up) + (u_a + u)  \cdot \na (\eps \,\pa_i \up) + \eps \,\pa_i (u_a + u)  \cdot \na \up     
 +  T^i  \frac{\na (\eps \,\pa_i \np)}{n_a + n}   \\ +  T^i   \eps \,\pa_i\left(\frac{1}{n_a + n}\right) \na \np   =  \eps \,\pa_i \na \fp + \eps \,\pa_i r_u.
\end{multline}
 
 We then take the scalar product with $\eta \, \eps \,\pa_i \up$ and make the sum in $i$. We find
 \begin{equation}
 \label{epsnablaup1}
\begin{aligned}
 & \sum_{i=1}^3\Big( \frac{1}{2} \frac{d}{dt} \int_{\R^3_+} \eta |\eps \, \pa_i \up|^2  -  \frac{1}{2} \int_{\R^3_{+}} \eta' \, (u_a+ u)_3 \,  |\eps \, \pa_i \up|^2  -  \frac{1}{2} \int_{x_3=0}  \, (u_a+ u)_3 \,  |\eps \, \pa_i \up|^2\Big)  \\
 & = I_1 +  I_2 + \mathcal{G} + \mathcal{S}
  \end{aligned}
\end{equation}
where 
\begin{equation*}
I_1 \: := \:  - \sum_{i=1}^3\Big( \int_{\R^3_+} \eta \,  T^i   \frac{\na (\eps \,\pa_i \np)}{n_a + n} \cdot \eps \,\pa_i \up\Big), \quad 
I_2 \: := \:  \sum_{i=1}^3 \int_{\R^3_+} \eta \, \eps \,\pa_i \na \fp \cdot \eps \,\pa_i \up.
\end{equation*}
 The term $\mathcal{S}$ above will refer from now and on  to  any admissible singular term,  that is a quantity which is singular but  bounded as 
 \begin{equation}
 \label{Sronddef} | \mathcal{S}| \leq  \mu \| \sqrt{\eta'} \eps \nabla \up \|_{L^2}^2 +  C\big\| \sqrt{\eta'} \big(\np, \up, \fp,  \eps \nabla \np, \eps \div \up, \eps \nabla \fp\big)\|_{L^2}^2.
 \end{equation}
 The first term in $\mathcal{S}$ will be absorbed by positive terms in the left hand side for $\mu$ small enough,  while the second term
  will be absorbed by the left hand sides of Estimates A and B. 
The term $\mathcal{G}$ above will refer from now and on to a good term, that is satisfying \eqref{Good}.

\medskip
 {\bf 1) Treatment of $I_1$.}
We use the identity: 
 \begin{equation}
 \label{epsnablaup2}
\begin{aligned}
I_1 =   \int_{\R^3_+} \eta \,  T^i   \frac{ (\eps \,\pa_i \np)}{n_a + n} \div( \eps \,\pa_i \up)  +  \int_{\R^3_+} \eta \,  T^i   (\eps \,\pa_i \np)\na\left(\frac{ 1}{n_a + n}\right) \cdot \eps \,\pa_i \up \\
 {+ \int_{\R^3_+} \eta' \,  T^i  \frac{ (\eps \,\pa_i \np)}{n_a + n}  \eps \,\pa_i \up_3 } + \int_{x_3=0}  \,  T^i  \frac{ (\eps \,\pa_i \np)}{n_a + n}  \eps \,\pa_i \up_3.
  \end{aligned}
\end{equation}

We shall now use the transport equation satisfied by $\eps \, \pa_i \np$, which reads:
\begin{equation}
(\pa_t  + (u_a +u)\cdot \na )(\eps \, \pa_i \np) + \eps \, \pa_i(n_a +n) \,  \div \up + (n_a +n) \,  \div (\eps \, \pa_i\up) 
+  \eps \, \pa_i(u_{a} + u )\cdot \nabla \np   = \eps \na r_n.
\end{equation}
Therefore, we have, relying on the standard manipulations:
 \begin{equation}
 \label{epsnablaup3}
\begin{aligned}
 &\int_{\R^3_+} \eta \,  T^i   \frac{ (\eps \,\pa_i \np)}{n_a + n} \div( \eps \,\pa_i \up)  = 
 - \frac{1}{2} \frac{d}{dt} \int_{\R^3_{+}} \eta \frac{T^i}{(n_a + n)^2}  |\eps \pa_i \np|^2  \\ 
  &  +   \frac{1}{2} \int_{\R^3_{+}} \eta' \, \frac{T^i(u_a + u)_3}{(n_a+n)^2} |\eps \pa_i \np|^2   +  \frac{1}{2} \int_{x_3=0}  \,\frac{T^i(u_a + u)_3}{(n_a+n)^2} |\eps \pa_i \np|^2  + \mathcal{G} + \mathcal{S}.
    \end{aligned}
\end{equation}
We end up with 
\begin{equation} \label{estimI1C}
\begin{aligned}
I_1 \: =  & \: - \frac{1}{2} \frac{d}{dt} \int_{\R^3_{+}} \eta \frac{T^i}{(n_a + n)^2}  |\eps \pa_i \np|^2  \\ 
  &  \: +   \frac{1}{2} \int_{\R^3_{+}} \eta' \, \frac{T^i(u_a + u)_3}{(n_a+n)^2} |\eps \pa_i \np|^2   +  \frac{1}{2} \int_{x_3=0}  \,\frac{T^i(u_a + u)_3}{(n_a+n)^2} |\eps \pa_i \np|^2  \\
&  \: {+ \int_{\R^3_+} \eta' \,  T^i  \frac{ (\eps \,\pa_i \np)}{n_a + n}  \eps \,\pa_i \up_3 } + \int_{x_3=0}  \,  T^i  \frac{ (\eps \,\pa_i \np)}{n_a + n}  \eps \,\pa_i \up_3 \: + \:  \mathcal{G} + \mathcal{S}.
\end{aligned}
\end{equation}

\medskip
 {\bf 2) Treatment of $I_2$.}
After some integration by parts, we can rewrite $I_2$ as 
 \begin{equation}
I_2   \: =  \:  - \sum_{i=1}^3 \left( \int_{\R^3_+} \eta \, \eps \,\pa_i  \fp \,  \eps \,\div  \pa_i \up + { \int_{\R^3_+}  \eta' \, \eps \,\pa_i  \fp \,  \eps \,  \pa_i \up_3} + \int_{x_3=0}  \, \eps \,\pa_i  \fp \,  \eps \,  \pa_i \up_3 \right).
\end{equation}
 Then we have
  \begin{equation}
  \label{derivu}
\begin{aligned}
 - \sum_{i=1}^3  \int_{\R^3_+} \eta \, \eps \,\pa_i  \fp \,  \eps \,\div  \pa_i \up  & =    \int_{\R^3_+} \eta \, \eps^2 \,\Delta  \fp \,   \,\div   \up  + {  \int_{\R^3_+} \eta' \, \eps \,\pa_{3}  \fp \,   \eps \,\div   \up} +  \int_{x_3=0}  \, \eps \,\pa_{3}  \fp \,   \eps \,\div   \up\\
&= L + \mathcal{S} + \mathcal{B}
  \end{aligned}
\end{equation}
where we  set 
\begin{equation}
\label{Ldef} 
L := \int_{\R^3_+} \eta \, \eps^2 \,\Delta  \fp \,   \,\div   \up.
\end{equation}
Similarly to $\mathcal{B}$ stands for an admissible boundary term which is bounded by 
   boundary terms in the left hand sides of Estimates A and B, that is to say
   $$ | \mathcal{B} | \leq  C \| \Gamma ( \np, \up, \eps\partial_{x_{3}} \fp, \eps \nabla \np, \eps \div \up) \|_{L^2(\mathbb{R}^2)}^2.$$
Hence, 
\begin{equation} \label{estimI2C}
I_2 \: = \:  - \sum_{i=1}^3  \left( { \int_{\R^3_+}  \eta' \, \eps \,\pa_i  \fp \,  \eps \,  \pa_i \up_3} + \int_{x_3=0}  \, \eps \,\pa_i  \fp \,  \eps \,  \pa_i \up_3 \right) \: +  \: L + \mathcal{S} + \mathcal{B}.
\end{equation}
Putting together \eqref{estimI1C}, \eqref{estimI2C} and \eqref{epsnablaup1}, we deduce 
\begin{equation} \label{epsnablaupbis}
\begin{aligned}
& \frac{1}{2} \frac{d}{dt} \int_{\R^3_+} \eta |\eps \, \na \up|^2 \: + \:   \frac{1}{2} \frac{d}{dt} \int_{\R^3_{+}} \eta \frac{T^i}{(n_a + n)^2}  |\eps \na \np|^2   \\ 
\le & \: \frac{1}{2} \int_{\R^3_{+}} \eta' \, (u_a+ u)_3 \,  |\eps \, \na \up|^2  +  \frac{1}{2} \int_{x_3=0}  \, (u_a+ u)_3 \,  |\eps \, \na \up|^2\Big) +   \frac{1}{2} \int_{\R^3_{+}} \eta' \, \frac{T^i(u_a + u)_3}{(n_a+n)^2} |\eps \na \np|^2   \\
&  +  \frac{1}{2} \int_{x_3=0}  \,\frac{T^i(u_a + u)_3}{(n_a+n)^2} |\eps \na \np|^2   \: {+ \int_{\R^3_+} \eta' \,  T^i  \frac{ (\eps \,\na \np)}{n_a + n} \cdot  \eps \,\na \up_3 } + \int_{x_3=0}  \,  T^i  \frac{ (\eps \,\na \np)}{n_a + n} \cdot  \eps \,\na \up_3 \\
& -   \left( { \int_{\R^3_+}  \eta' \, \eps \,\na \fp \cdot  \eps \,  \na \up_3} + \int_{x_3=0}  \, \eps \,\na  \fp \, \cdot  \eps \,  \na \up_3 \right)
 + \:  L +   \mathcal{G} + \mathcal{S}  + \mathcal{B}.
\end{aligned}
\end{equation}

\medskip
{\bf 3) Treatment of $\I$}. We recall that the  $\I$ term defined in Lemma \ref{lem+} can be written as 
\begin{align*} 
\I &= -  \int_{\R^3_{+}} \eta \,  \fp \, \div \big( (n_{a}+ n)  \up \big) - { \int_{\R^3_{+}}  \eta' \, \fp \,  \big( (n_{a}+ n)  \up_3 \big) }\\
&    := J +  \mathcal{G}+ \mathcal{S}.
\end{align*} 
with $J := -  \int_{\R^3_{+}} \eta \,  \fp \, (n_{a}+ n)  \div \up$. 
We shall study $\I$ (and $J$) in a different way than that followed for the first two energy estimates. The idea is to combine $J$ with $I_2$ through the term $L$ (see \eqref{Ldef}-\eqref{estimI2}). Therefore, we decompose $J$ as follows: 
 \begin{equation}
\begin{aligned}
  J &= \int_{\R^3_{+}} \eta \,  \div \up \Big( \eps^2 \,\Delta  \fp - \fp \, (n_{a}+ n) \Big) - L \\
   &=   \int_{\R^3_{+}} \eta \,  \div \up \Big( \eps^2 \,\Delta  \fp -  e^{-\phi_a} \fp (  1 + h(\phi) )  \Big) +   \int_{\R^3_{+}} \eta \,  \div \up \, \fp \Big(  e^{-\phi_a}  (  1 + h(\phi) ) -  (n_{a}+ n) \Big) - L\\
 & =: J_1+J_2 - L. 
  \end{aligned}
\end{equation}
Since  by construction of the approximation solution, we have $e^{-\phi^0}= n^0$, $J_2$ can be bounded by:
$$
J_2 \: \le \: \mu  \int_{\R^3_{+}} \eta' \, ((\eps \, \div \up)^2 + \fp^2) + \mathcal{G}.
$$
Considering $J_1$, by the Poisson equation, we have:
\begin{equation}
\begin{aligned}
J_1 = \int_{\R^3_{+}} \eta \,  \div \up \, \np +  \int_{\R^3_{+}} \eta \,  \div \up \,  r_\phi.
  \end{aligned}
\end{equation}
One can rely on the transport equation satisfied by $\np$ to replace $\div \up$, which yields:
\begin{equation}
\begin{aligned}
 \int_{\R^3_{+}} \eta \,  \div \up \, \np &=  \int_{\R^3_{+}} \eta \,  \, \frac{\np}{n_a + n} \left( - \pa_t  \np - (u_{a} + u )\cdot \nabla \np    + r_n \right),
  \end{aligned}
\end{equation}
which we shall recast, after the usual manipulations, as
\begin{equation*}
\begin{aligned}
 \int_{\R^3_{+}} \eta \,  \div \up \, \np = - \frac{1}{2} \frac{d}{dt} \int_{\R^3_{+}} \eta \,  \, \frac{\np^2}{n_a + n} + \frac{1}{2} \int_{\R^3_{+}} \eta' \, \frac{(u_a+ u)_3}{n_a+n} \,  \np^2  + \frac{1}{2} \int_{x_3=0} \, (u_a+ u)_3 \,  \np^2 \\
+ \mu \int_{\R^3_{+}} \eta' \, (\np^2 + |\up|^2) + C_a  \int_{\R^3_{+}} \eta \, \np  r_n.
  \end{aligned}
\end{equation*}
We obtain 
\begin{equation} \label{estimIC}
\begin{aligned}
\I \: \le \: & -L \: + \:  - \frac{1}{2} \frac{d}{dt} \int_{\R^3_{+}} \eta \,  \, \frac{\np^2}{n_a + n} + \frac{1}{2} \int_{\R^3_{+}} \eta' \, \frac{(u_a+ u)_3}{n_a+n} \,  \np^2  + \frac{1}{2} \int_{x_3=0} \, (u_a+ u)_3 \,  \np^2 \\ 
& + \mathcal{G} \: + \: \mathcal{S}. 
\end{aligned}
\end{equation}

\medskip
Consequently, by combining  \eqref{epsnablaupbis} and \eqref{estimIC}, we obtain
\begin{equation}
\begin{aligned}
\label{estimationC1}
\I  +  \frac{1}{2} \frac{d}{dt} \int_{\R^3_+} \eta\Big(  |\eps \, \nabla \up|^2 +  \frac{T^i}{(n_a + n)^2}  |\eps \na \np|^2 
+\frac{\np^2}{n_a + n} \Big)\\
\leq  \frac{1}{2} \int_{\R^3_{+}} \eta' \, (u_a+ u)_3 \,  |\eps  \nabla \up|^2 +   \frac{1}{2} \int_{\R^3_{+}} \eta'  T^i \,\frac{(u_a + u)_3}{(n_a+n)^2} |\eps \na \np|^2 \\
+ \frac{1}{2} \int_{\R^3_{+}} \eta' \, \frac{(u_a+ u)_3}{n_a+n} \,  \np^2 -
\int_{\R^3_+}  \eta' \, \eps \,\na  \fp \, \cdot  \eps \,  \na \up_3 
+  \int_{\R^3_+} \eta' \,  T^i  \frac{ (\eps \,\na \np)}{n_a + n}  \cdot \eps \,\na \up_3 \\
 + B.T.+   \mathcal{G}+ \mathcal{S} + \mathcal{B}.
  \end{aligned}
\end{equation}
Here $B.T.$ denotes the important  boundary contributions:
$$ B.T.= \int_{x_{3}= 0}\Big( -  \eps \partial_{3} \fp \, \eps \partial_{3} \up + { T^{i} \over n_{a} + n  } \eps \nabla \np \cdot \eps \nabla \up_{3}  + 
 { 1 \over 2 } T^i { ( u_{a} + u)_{3} \over (n_{a} + n)^2} | \eps \nabla \np |^2  + { 1 \over 2 } (u_{a} + u)_{3} | \np |^2. $$

\medskip

 {\bf 3) Conclusion of estimate C.}
 
 By combining, \eqref{estimationC1} and Lemma \ref{lem+}, we obtain
 
\begin{equation}
\begin{aligned}
\frac{d}{dt}  \int_{\R^3_+}  \eta  \Big(  (n_a + n) \frac{|\up|^2}{2} +   \frac{ T^{i} |\np|^2}{2 (n_a + n) }  +  \frac{1}{2}\ |\eps \, \partial \up|^2  
+ \frac{1}{2}\frac{T^i}{(n_a + n)^2}  |\eps \na \np|^2  +  \frac{\np^2}{n_a + n} 
\Big) \\
+ \frac{1}{2} \int_{\R^3_+}  \eta' \, Q^0(\np,\up) + \frac{1}{2} \int_{\R^3_+}  \eta' \, \big(Q^C(\eps \, \na \np, \eps \, \na \up_3) + |(u_{a} + u)_{3}| \, |\eps \nabla (u_{1}, u_{2})|^2\big) \\
+ \frac{1}{2}\int_{x_3=0} Q^0(\np,\up_3) + \frac{1}{2}\int_{x_3=0} Q^C(\eps \, \na \np, \eps \, \na \up_3)   \\
 \leq  -\int_{\R^3_+}  \eta' \, \eps \,\na  \fp \, \cdot  \eps \,  \na \up_3 -  \int_{x_{3}= 0}  \eps \partial_{3} \fp \, \eps \partial_{3} \up_{3} 
 + \mathcal{G} + \mathcal{S} + \mathcal{B}.
\end{aligned}
\end{equation}
where  $Q^C$ is the quadratic form of the  symmetric matrix:
$$
M_C = \Gamma  \begin{bmatrix} T^i\frac{|u_{a,3}|}{n_a^2} Id_{3}&- \frac{T^i}{n_a} Id_{3}\\  -\frac{T^i}{n_a} Id_{3} &|u_{a,3}| Id_{3} \end{bmatrix} - \mu I_6
$$
which is positive thanks to  the Bohm condition \eqref{hypbohm}.

 To conclude,  we can handle the two first terms in the right hand side of the above estimates by using the Young inequality.
  Indeed,  we write
  $$ \Big|  \int_{\R^3_+}  \eta' \, \eps \,\na  \fp \, \cdot  \eps \,  \na \up_3 \Big| \leq {\tilde \mu   \over 2} \| \sqrt{\eta'} \eps \nabla \up_{3} \|_{L^2}^2 + { 1 \over 2 \tilde \mu}
   \| \sqrt{\eta'} \eps \nabla \fp \|_{L^2}^2$$
    and for $\tilde \mu$ sufficiently small, the term  $\| \sqrt{\eta'} \eps \nabla \up_{3} \|_{L^2}^2$ can be absorbed in the left hand side since $Q^C$
     is positive while the other term can be incorporated in $ \mathcal{S}$. We proceed in the same way for the boundary term by writing
     $$ \Big|  \int_{x_{3}= 0}  \eps \partial_{3} \fp \, \eps \partial_{3} \up_{3}  \Big| \leq {\tilde \mu  \over 2} \| \eps \partial_{3} \up \|_{L^2(\mathbb{R}^2)}^2
      + { 1 \over 2 \tilde \mu}  \| \eps \partial_{3} \fp \|_{L^2(\mathbb{R}^2)}$$
       and we absorb the first term in the left hand side by the positivity of $Q^C$ while the other term can be incorporated in $\mathcal{B}$.
       
        To obtain \eqref{Cfinal}, it suffices  to observe that the term $\mu \| \sqrt{\eta'} \eps \nabla \up \|_{L^2}^2$ can be absorbed in the
         left hand side for $\mu$ sufficiently small and to integrate in time.


\bigskip
{\bf D.  End of the proof of Proposition \ref{propL2}}

 We can first   combine estimates \eqref{Afinal} and \eqref{Bfinal} in the following way. We consider  $ \eqref{Bfinal} +  \epsilon \eqref{Afinal}$ with $\epsilon$
 fixed  sufficienly small (independently of the other involved parameters $\mu$ and $\eps$) so that  the singular  term
   $ \| \sqrt{\eta'} (\eps \nabla \np, \eps \div \up) \|_{L^2_{T} L^2(\mathbb{R}^3_{+})}^2$ in the right hand side  of \eqref{Afinal} can be absorbed by the left hand side
    of \eqref{Bfinal}. This yields
  \begin{multline}
  \label{Dfinal1}
  \|(\np, \up, \fp,  \eps \nabla \np, \eps \div \up, \eps \nabla \fp)\|_{L^\infty_{T}L^2(\mathbb{R}^3_{+})}^2 + 
     \| \sqrt{\eta'}(\np, \up, \fp, \eps \nabla \np, \eps \div \up,  \eps \nabla \fp) \|_{L^2_{T} L^2(\mathbb{R}^3_{+})}^2  +   \\ 
  + \|\Gamma ( \np,  \up,  \eps \nabla \np, \eps \div \up, \eps \partial_{x_{3}} \fp )\|_{L^2_{T}L^2(\mathbb{R}^2)}^2 \\
   \leq C(C_{a}, M) \Big( \|(\np, \up, \fp, \eps \nabla \np, \eps \nabla \up,  \eps \nabla \fp)(0)\|_{L^2(\mathbb{R}_{+}^3)}^2 
    + \mu   \|  \sqrt{\eta'} \eps \nabla u \|_{L^2_{T}L^2(\mathbb{R}^3_{+})}^2 + \
    \\+  \|(\np, \up, \eps \nabla \np, \eps \nabla \up , \fp, \eps \nabla \fp)\|_{L^2_{T}L^2(\mathbb{R}^3_{+})}^2 
     + \| R^r \|_{L^2_{T}L^2(\mathbb{R}^3_{+})}^2 + \mu \| \sqrt{\eta'} R^s \|_{L^2_{T}L^2(\mathbb{R}^3_{+})}^2
    \Big).
 \end{multline}
   To handle the singular term $ \mu   \|  \sqrt{\eta'} \eps \nabla u \|_{L^2_{T}L^2(\mathbb{R}^3_{+})}^2$ in the right hand side of  \eqref{Dfinal1}, 
    we  consider $ \eqref{Dfinal1} + \sqrt{\mu} \eqref{Cfinal}$. This yields
    \begin{multline}
  \label{Dfinal2}
  \|(\np, \up, \fp,  \eps \nabla \np, \eps \div \up, \eps \nabla \fp)\|_{L^\infty_{T}L^2(\mathbb{R}^3_{+})}^2 +  \sqrt{\mu} \| \eps \nabla \up \|_{L^\infty_{T} L^2(\mathbb{R}^{3}_{+})}^2  \\
    \| \sqrt{\eta'}(\np, \up, \fp, \eps \nabla \np, \eps \div \up,  \eps \nabla \fp) \|_{L^2_{T} L^2(\mathbb{R}^3_{+})}^2 + \sqrt{\mu} \| \sqrt{\eta'} \eps \nabla \up \|_{L^2_{T} L^2(\mathbb{R}^3_{+})}^2  \\ 
  + \|\Gamma ( \np,  \up,  \eps \nabla \np, \eps \div \up, \eps \partial_{x_{3}} \fp )\|_{L^2_{T}L^2(\mathbb{R}^2)}^2 + \sqrt{\mu} \| \Gamma \eps \nabla \up \|_{L^2_{T} L^2(\mathbb{R}^2)}^2   \\
   \leq C(C_{a}, M) \Big( \|(\np, \up, \fp, \eps \nabla \np, \eps \div \up,  \eps \nabla \fp)(0)\|_{L^2(\mathbb{R}_{+}^3)}^2  \\
    + \mu   \|  \sqrt{\eta'} \eps \nabla u \|_{L^2_{T}L^2(\mathbb{R}^3_{+})}^2 + \sqrt{\mu}   \| \sqrt{\eta'}(\np, \up, \fp, \eps \nabla \np, \eps \div \up,  \eps \nabla \fp) (t)\|_{L^2_{T} L^2(\mathbb{R}^3_{+})}^2
    \\
    + \sqrt{\mu} \|\Gamma ( \np,  \up,  \eps \nabla \np, \eps \div \up, \eps \partial_{x_{3}} \fp )\|_{L^2_{T}L^2(\mathbb{R}^2)}^2  \\
    +  \|(\np, \up, \eps \nabla \np, \eps \nabla \up , \fp, \eps \nabla \fp)\|_{L^2_{T}L^2(\mathbb{R}^3_{+})}^2 
     + \| R^r\|_{L^2_{T}L^2(\mathbb{R}^3_{+})}^2 + \mu \|\sqrt{\eta'} R^s \|_{L^2_{T}L^2(\mathbb{R}^3_{+})}^2
    \Big).
 \end{multline}
  To conclude, we observe that by taking $\mu $ sufficiently small, the singular terms 
  \begin{multline}   \mu   \|  \sqrt{\eta'} \eps \nabla u \|_{L^2_{T}L^2(\mathbb{R}^3_{+})}^2 + \sqrt{\mu}   \| \sqrt{\eta'}(\np, \up, \fp, \eps \nabla \np, \eps \div \up,  \eps \nabla \fp) (t)\|_{L^2_{T} L^2(\mathbb{R}^3_{+})}^2  \\+ \sqrt{\mu} \|\Gamma ( \np,  \up,  \eps \nabla \np, \eps \div \up, \eps \partial_{x_{3}} \fp )\|_{L^2_{T}L^2(\mathbb{R}^2)} 
  \end{multline}
   can be absorbed in the right hand side.

We end the proof by integrating in time, by applying a Gronwall estimate and by using \eqref{hypmin}. 
\end{proof}

\section{Nonlinear stability}
In this section, we shall prove our main Theorem \ref{theomain}. As already mentioned, we shall actually prove a more precise version, see Theorem~\ref{main}.

Let us write the solution $(n^\eps, u^\eps, \phi^\eps)$ of \eqref{EP} under the form 
\begin{equation*}
n^\eps  \: = \:  n_a \: +  n , \quad u^\eps  \: = \: u_a  \: +  u , \quad \phi^\eps  \: = \:  \phi_a \: + \phi
\end{equation*}
where $n_a, u_a, \phi_a$ are shorthands for $n^\eps_{app}, u^\eps_{app}, \phi^\eps_{app}$ (defined in \eqref{ansatz}). 
Then, we get for $(n, u, \phi)$ the system
\begin{equation} \label{EPP}
\left\{
\begin{aligned}
& \pa_t  n +  (u_a + u ) \cdot \nabla n + n \,\div (u + u_{a})    + \div  (n_{a} u)  = \eps^{K} R_n,  \\
& \pa_t u + (u_a + u)  \cdot \na u + u \cdot \nabla u_{a}  + T^i   \left(\frac{\na n}{n_a +  n } - \frac{\na n_a}{n_a} \left( \frac{n}{n_a +  n} \right) \right) =  \na \phi + \eps^{K}R_u, \\
& \eps^2 \Delta \phi = n - e^{-\phi_a}( e^{ - \phi} - 1 \big) + \eps^{K+1} R_\phi.
\end{aligned}
\right.
\end{equation}
together with the boundary conditions
\begin{equation} \label{EPPbc1}
 \phi\vert_{x_3=0} = 0
\end{equation}
and the initial condition
\begin{equation} \label{EPPbc2}
u\vert_{t = 0} = \eps^{K+ 1} u_{0}, \quad n\vert_{t=0} = \eps^{K+ 1} n_{0}.
\end{equation}
Observe here that in \eqref{EPP}, $\eps^{K} R_n,  \eps^{K}R_u$ and $\eps^{K+1} R_\phi$ are remainders that appear because $(n_a,u_a,\phi_a)$ is not an exact solution of \eqref{EP}.

\bigskip

The main result of this section is 
\begin{theorem}
\label{main}
Let $m \geq 3$,  $(n_{0}, u_{0})\in H^m(\R^3_{+})$. Let $K \in \N^*, K \geq m$ and $(n_{a}, u_{a}, \phi_{a})$
 an  approximate solution at order $K$, given by  Theorem \ref{deriv}, which is defined on $[0, T_{0}]$. There exists
 $\eps_0>0$ and $\delta_{0}>0$  such that for every $\delta \in (0, \delta_{0}]$ and  for every $\eps \in(0, \eps_{0}]$, there is $C>0$ independent of $\eps$ such that the solution of  \eqref{EPP}-\eqref{EPPbc1}-\eqref{EPPbc2} is defined  on $[0,T_{0}]$
  and satisfies the estimate
  $$ \eps^{|\alpha|} \| \partial^\alpha(n,u, \phi, \eps\nabla  \phi) \|_{L^2(\mathbb{R}^3_{+})} \leq C \eps^{K}, \quad \forall t \in [0, T_0], \quad \forall \alpha
   \in \mathbb{N}^3, \, |\alpha| \leq m.$$ 
  
\end{theorem}
We thus obtain:

\begin{corollary}
\label{coro}
Under the assumptions of Theorem \ref{main}, we have 
  \begin{equation}
   \left\| \Big(n^\eps-n_{a},u^\eps-u_{a}, \phi^\eps-\phi_{a} \Big)\right\|_{H^{m}(\mathbb{R}^3_{+})} \leq C \eps^{K-m}, \quad \forall t \in [0, T_0].
   \end{equation}
  In particular, we get the $L^2$ and $L^\infty$ convergences as $\eps \rightarrow 0$:
 \begin{equation}
  \begin{aligned}
&\sup_{[0,T_0]}\left(  \left\| n^\eps- n^0 - N^0\left(\cdot,\cdot,\frac{\cdot}{\eps}\right)\right\|_{L^\infty(\R^3_+)}  +   \| n^\eps- n^0 \|_{L^2(\R^3_+)}\right) \rightarrow 0, \\
&\sup_{[0,T_0]}  \left(    \| u^\eps- u^0 - U^0\left(\cdot,\cdot,\frac{\cdot}{\eps}\right) \|_{L^\infty(\R^3_+)}  +   \| u^\eps- u^0 \|_{L^2(\R^3_+)}\right) \rightarrow 0, \\
& \sup_{[0,T_0]}  \left(    \left\| \phi^\eps- \phi^0 - \Phi^0\left(\cdot,\cdot,\frac{\cdot}{\eps}\right) \right\|_{L^\infty(\R^3_+)}   +   \| \phi^\eps- \phi^0 \|_{L^2(\R^3_+)}\right) \rightarrow 0.
  \end{aligned}
  \end{equation}
\end{corollary}
Then Theorem \ref{theomain} is a straightforward consequence of this corollary.

\bigskip

From now on, our goal is to prove Theorem \ref{main}.

For $\eps >0$ fixed, since  in the equation \eqref{EPP}, the  term involving $\phi$ in the second equation can be considered
as a semi-linear term,  the  known local existence results for the compressible Euler equation with strictly dissipative boundary conditions (\cite{BenzoniSerre}) can be applied to 
 the system \eqref{EPP}. 
  Let us assume that  $(n_{0}, u_{0}) \in H^m(\R^3_{+})$ for $m \geq 3$, 
   then there exists $T^\eps >0$ and  a unique solution of \eqref{EPP} defined on $[0,T^\eps]$ such that
     $ u \in \mathcal{C}([0, T^\eps), H^m)$ and that there exists $M>0$ independent of $\eps$ 
     such that the assumptions \eqref{hypmin} and \eqref{hypbohm} of Proposition \eqref{propL2} are satisfied
  with 
\begin{equation} \label{h0h1}
h_0(\phi) \: := \:  -{ e^{-\phi} - 1 + \phi \over \phi}, \quad h_1(\phi) \: := \:  e^{-\phi}- 1.  
\end{equation}

Thanks to the well-posedness in $H^m$ for $m \geq 3$ of the system \eqref{EPP}, we can define
$$ T^\eps= \sup \big\{ T \in [0,  T_{0}], \quad   \forall t \in [0,T],  \|(n,u, \phi)\|_{H^m_{\eps}(\R^3_{+})} \leq \eps^r \big\}$$
where $r$ is chosen such that
\begin{equation}
\label{choixr}
 5/2<r<K
 \end{equation}
 and  the $H^m_{\eps}$ norm is defined by 
 $$ \|f\|_{H^m_{\eps}(\R^3_{+})} = \sum_{ | \alpha | \leq m} \eps^{|\alpha|} \| \partial_{x_{1}}^{\alpha_{1}} \partial_{x_{2}}^{\alpha_{2}} \partial_{x_{3}}^{\alpha_{3}} f \|_{L^2(\R^3_{+})}.$$

   The difficulty is thus to prove that  the solution actually exists on an interval of time independent of $\eps$.
    We shall  get this result  by proving  uniform energy estimates combined with the previous local existence result when  the initial data
     and the source term are sufficiently small (i.e. when the approximate solution $ (n_{a}, u_{a})$ is sufficiently accurate).
     
     For $m \geq 3$, we shall prove by using a priori  estimates  that  under the assumptions of Theorem \ref{main}, we have 
  for    every $T\in [0, T^\eps]$ the control
 \begin{equation}
 \label{Hcmeps}
 \|\big(n,u, \phi ) (T) \|_{H^m_{\eps}(\R^3_{+})} \leq C(C_{a}, M) \eps^{K} e^{ T { C(C_{a}}, M) \over \sqrt{\mu}}
 \end{equation}
 with  $C(C_{a}, M)$ independent of $\eps$,  $ \mu$ and $T$ for $ \eps \in (0, 1], $ $\mu \in (0, 1]$ and $T \in [0, T_{0}]$. 
 
 To prove this estimate, we can apply the operator $ \ZZ^\alpha:=  (\eps \pa)^\alpha$ to \eqref{EPP} for $| \alpha | \leq m-1.$ We obtain for
  $ \ZZ^\alpha(n,u, \phi)$ the system
\begin{equation} \label{linear2}
\left\{
\begin{aligned}
 \pa_t  \ZZ^\a n +  (u_{a}+ u )\cdot \nabla \ZZ^\a n + (n_{a}+ n) \div \ZZ^\a u  =\mathcal{C}_{n}  + \eps^K \ZZ^\alpha R_{n} & , \\ 
 \pa_t \ZZ^\a u + (u_a + u ) \cdot \na \ZZ^\a u  +  T^i   \frac{\na \ZZ^\a n}{n_a + n}  \\
 =  \na \ZZ^\a \phi + \mathcal{C}_u + \eps^K \ZZ^\a R_{u} & , 
 \\
 \eps^2 \Delta \ZZ^\a \phi = \ZZ^\a n + e^{-\phi_a} \ZZ^\a \phi( 1  + h)  + \mathcal{C}_\phi+ \eps^{K+1} \ZZ^\a R_{\phi}& .
\end{aligned}
\right.
\end{equation}
One has $h = h_0$  for $|\alpha| = 0$ and $h=h_1$ for $|\alpha| \ge 1$.  The functions $\mathcal{C}_n,\mathcal{C}_u$ and $\mathcal{C}_\phi$ are remainders mostly due to commutators; in $\mathcal{C}_n$, we include  the term $-\ZZ^\a [u \cdot \nabla n_a] - \ZZ^\a [n \, \div u_{a}] $ as well, while in $\mathcal{C}_u$ we also include $ -\ZZ^\a [u  \cdot \na u_a]  + T^i    \ZZ^\a\left[ \frac{\na n_a}{n_a} \left( \frac{ n}{n_a + n}  \right)\right]$.
One can observe that this corresponds to the ``abstract'' system that we have studied in Proposition \ref{propL2}.

The assumptions \eqref{hypmin}, \eqref{hypbohm} of Proposition \ref{propL2} are matched on $[0, T^\eps]$ for $\eps$ sufficiently small
 since they are verified by the approximate solution. The only estimate that does not follow directly from the definition of $T^\eps$ is
  the estimate of $\| \partial_{t}(n, \phi) \|_{L^\infty}$.
   For $\partial_{t}n$, we immediately get by using the first line of \eqref{EPP} that
   $$ \| \partial_{t} n\|_{L^\infty} \leq C_{a} ( \eps^{ r-1} + \eps^{K}) \leq M,$$
    for $M$ sufficiently small. In a similar way,  by taking the time derivative in  the elliptic equation for  $\phi$, we have
    \begin{equation}
    \label{poissondt} - \eps^2 \Delta \partial_{t} \phi  + e^{- (\phi_{a} + \phi)} \partial_{t} \phi =  - \partial_{t}n +  \partial_{t} (e^{-\phi_{a}}) ( e^\phi - 1) + \eps^{K+1} \partial_{t} R_{\phi}.
    \end{equation}
     By using the maximum principle for this elliptic equation, we thus get that
     $$  \|\partial_{t} \phi \|_{L^\infty} \leq C(C_{a}, M) \big( \| \partial_{t}n \|_{L^\infty} + \| \phi\|_{L^\infty} + \eps^{K+ 1}\big) \leq C(C_{a}, M) \eps^{r-1}.$$

 To use the result of Proposition \ref{propL2}, we also   need to check that the remainders $\mathcal{C}_{n}, \,  \mathcal{C}_{u}$  can be split as in  the assumption \eqref{hypsource}.

  To describe the structure of these remainders, we can start   with the  study of  $\mathcal{C}_{n}$. We can distinguish between a linear part which corresponds to
  $$ \mathcal{C}_{n}^{l}= - [ \ZZ^\alpha, u_{a}] \cdot \nabla n -  \ZZ^\alpha(n \, \div u_{a}) - [\ZZ^\alpha,  n_{a}] \div u -  \ZZ^\alpha( u \cdot \na  n_{a})$$
   and a nonlinear part which is
   $$ \mathcal{C}_{n}^{nl}= -  [\ZZ^\alpha, u] \cdot \nabla n -  [\ZZ^\alpha, n ] \div u.$$
   By using standard tame estimates in Sobolev spaces, the  nonlinear part can be seen as a regular part in the decomposition \eqref{hypsource}. Indeed,  we  easily get that for some $C>0$ independent of $\eps$, we have
  $$ \| \mathcal{C}_{n}^{nl} \|_{H^1_{\eps}} \leq C \big( \|\nabla u \|_{L^\infty}  +  \| \nabla n \|_{L^\infty}  \big)   \| (n,u) \|_{H^{| \alpha |+ 1}_{\eps}}$$
   and hence by using the Sobolev embedding,  we obtain that on  $[0, T^\eps]$,
   $$ \| \mathcal{C}_{n}^{nl} \|_{H^1_{\eps}} \leq C \eps^{r- {5 \over 2}}  \|(n, u) \|_{H^m_{\eps}}.$$
    
       To estimate the linear part $\mathcal{C}_{n}^{l}$, we can use  the Leibnitz formula and the estimates \eqref{boundnuf}, \eqref{boundna} on the approximate solution.
    By using the $\eps$ weighted derivatives, the  terms that lead to singular terms
     are  the ones of   the first type
     $$  \eps^k \partial^{\beta} n_{a} \cdot \nabla \partial^\gamma u$$
      or the ones of the second type
      $$  \eps^k \partial_{x_3} \partial^\beta n_{a} \cdot \partial^\gamma u$$
      with $| \beta |+ | \gamma |= k.$
       Indeed, for the terms of the second type, 
       we obtain a singular term as soon as $\beta$ is made only with  $x_{3}$ derivatives and when the all  hit the boundary layer term in $n_{a}$, in this case, we can use 
        \eqref{betagrand} to estimate  them by 
 \begin{equation}
 \label{commu}
  C_{a} \big( \mu \,\eta' |\langle \eps \partial \rangle^m (n,u) | +     |\langle \eps \partial \rangle^m (n,u) |\big)
  \end{equation}
        with the notation
        $$  |\langle \eps \partial \rangle^m (n,u) |  = \sum_{| \beta | \leq m} | (\eps \partial)^\beta (n,u) |.$$
        We thus  see the first term in \eqref{commu} as a singular part of the source term and the second one as a regular part in the decomposition \eqref{hypsource}.

       For the  terms of the first type, in order to estimate them  with   $u$ in a $ H^{m}_{\eps}$ space, we need to write them under the form
      $$ \eps^{|\beta|- 1} \partial^\beta n_{a} \cdot (\eps \nabla ) (\eps \partial)^\gamma u$$
       and hence we obtain again a singular term as soon as   $\partial^\beta$  is made only of $x_{3}$ derivatives and that they all hit the boundary layer terms in $n_{a}$.  By using~\eqref{betagrand}, 
       we can estimate this contribution  again by 
                 $$   C_{a} \big( \mu \,\eta' |\langle \eps \partial \rangle^m (n,u) | +     |\langle \eps \partial \rangle^m (n,u) |\big)$$
        
        We thus obtain an estimate
     $$   |\mathcal{C}^l| + |\eps \nabla \mathcal{C}^l| \leq  C_{a} \big( \mu \,\eta' |\langle \eps \partial \rangle^m (n,u) | +     |\langle \eps \partial \rangle^m (n,u) |\big).$$
        
        We can proceed in the same way to  handle $\mathcal{C}_{u}$. For example, for the commutator 
        $$ \mathcal{C}^p:=  [ \ZZ^\alpha,{ 1 \over n_{a} + n}] \na n, $$ 
        we can expand
        $$ (\mathcal{C}^p, \eps \nabla \mathcal{C}^p)= \sum_{ | \beta| +  | \gamma | \leq m, \, \beta \neq 0}  \star_{\beta, \gamma} \big( \ZZ^\beta n_{a} \nabla \ZZ^\gamma n + \ZZ^\beta n   \nabla \ZZ^\gamma  n\big)$$
         where  $\star_{\beta, \gamma}$ stand for terms which are uniformly bounded in $L^\infty$ on $[0, T^\eps]$. By using the same arguments as before, we thus get that
         $$ |  \mathcal{C}^p| +  |\eps \nabla \mathcal{C}^p | \leq  C(C_{a}, M) \big( \mu \,\eta' |\langle \eps \partial \rangle^m (n,u) | +     |\langle \eps \partial \rangle^m (n,u) |\big)
          + C(C_{a},  M)  |  \tilde{ \mathcal{C}}^p |$$
          where 
         $  |  \tilde{ \mathcal{C}}^p | $ is bounded in $L^2$ by
         $$  C(C_{a}, M) \eps^{ r- { 5 \over 2} } \| n \|_{H^m_{\eps}}$$
          by using the  usual product estimates in Sobolev spaces.

         For $ \eps^{-1}\mathcal{C}_{\phi}$, we can use again the bounds on the approximate solution \eqref{boundnuf}, \eqref{boundna} and standard tame estimates in Sobolev spaces.
          Indeed, we have that $ \eps^{-1}\mathcal{C}_{\phi}$ can be expanded as a sum of terms either of the form  
          $$ \eps^{-1} (\ZZ^\alpha e^{- \phi_{a}}  )(e^{-\phi}-  1) $$
           which is bounded by 
           $ C(C_{a}, M) \big( \mu \,\eta' | \phi  | + | \phi|)$
           or of the form 
         $$  \eps^{-1}  \ZZ^\beta ( e^{-\phi_{a}}) \, \ZZ^\gamma \phi  \, e^{- \phi}$$
          with $ | \beta | + | \gamma|  \leq m$ which can also be bounded by 
          $ C(C_{a}, M) \big( \mu \,\eta' | \phi  | + | \phi|)$
          or
          $$  \eps^{-1}  \ZZ^\beta ( e^{-\phi_{a}}) \, \ZZ^{\gamma_{1} } \phi  \cdots  \ZZ^{\gamma_{r}} \phi \, e^{-\phi}$$
           with $ r \geq 2$,  $| \beta| + | \gamma_{1}| +  \cdots +  |  \gamma_{r}| \leq m$ and $\beta \neq 0$.
            Since this term is at least quadratic in $\phi$, we can use tame estimates to 
            bound it  in $L^2$ by 
            $$ \eps^{-1} C(C_{a}, M) \| \phi \|_{L^\infty} \| \phi \|_{H^{m-1}_{\eps}} \leq C(C_{a}, M) \eps^{ r- { 5 \over 2 }} \| \phi \|_{H^{m-1}_{\eps}}.$$
         We thus get that we can split $\eps^{-1} \mathcal{C}_{\phi}$ according to  \eqref{hypsource}:
         $$ | \eps^{-1} \mathcal{C}_{\phi} | \leq C(C_{a}, M ) \big( \mu \eta' | \langle \eps \pa \rangle^{m-1} \phi | + | \tilde{ \mathcal{C}}_{\phi}| \big)$$
          with
          $$ \| \tilde{ \mathcal{C}}_{\phi}\|_{L^2} \leq C(C_{a}, M ) \| \phi \|_{H^{m-1}_{\eps}}.$$
          
           By using the above expansions of $\mathcal{C}_{\phi}$, we also easily obtain that
           $$ | \partial_{t} \mathcal{C}_{\phi}| \leq C(C_{a}, M) \big(  \eps \eta' | \langle  \eps \pa  \rangle^{m-1} \partial_{t} \phi | + | \mathcal{C}_{\phi}^1|$$
           where $\mathcal{C}_{\phi}^1$ can be estimated as
          $$ \|{ \mathcal{C}}_{\phi}^1\|_{L^2} \leq C(C_{a}, M )  \eps\| \pa_{t} \phi \|_{H^{m-1}_{\eps}}.  $$  
          
          In summary, we have proven that the commutator terms can be split under the form
         $$ \big|  \big(\mathcal{C}_{n}, \mathcal{C}_{u}, \eps \nabla \mathcal{C}_{n}, \eps \nabla  \mathcal{C}_{u},  { \mathcal{C}_{\phi} \over \eps}, \partial_{t} \mathcal{C}_{\phi}\big) |
          \leq C(C_{a}, M) \big(\mu \eta' \mathcal{C}^s +\mathcal{ C}^r \big)$$
           with
           \begin{align*}
           &  \| \sqrt{\eta'} \mathcal{C}^s  \|_{L^2} \leq C(C_{a}, M)  \big( \|\sqrt{\eta'} (n,u, \phi)\|_{H^{m}_{\eps}}  +    \eps    \| \sqrt{\eta'} \partial_{t}\phi\|_{H^{m-1}_{\eps}} \big),  \\
            &   \| \mathcal{C}^r \|_{L^2} \leq C(C_{a}, M) \big(  \| (n, u, \phi) \|_{H^m_{\eps}} + \eps \| \partial_{t} \phi \|_{H^{m-1}_{\eps}} \big).
            \end{align*}

Consequently, by using Proposition \ref{propL2}, we obtain for every $T \in [0, T^\eps]$
\begin{multline}
\label{fin1}
 \sqrt{\mu} \big(\|(n,u,\phi)\|_{L^\infty_{T} H^m_{\eps}}^2 + \| \sqrt{\eta'} (n,u, \phi) \|_{L^2_{T}H^m_{\eps}}^2 \big)
 \leq C(C_{a}, M)  \Big(  \eps^{2K} +   \mu    \|\sqrt{\eta'} (n,u, \phi)\|_{L^2_{T}H^{m}_{\eps}}^2   \\+ \mu    \| \sqrt{\eta'}\, \eps \partial_{t}\phi\|_{L^2_{T}H^{m-1}_{\eps}}^2
  +  \| (n, u, \phi) \|_{L^2_{T}H^m_{\eps}}^2 + \| \eps  \partial_{t} \phi \|_{L^2_{T}H^{m-1}_{\eps}}^2 \Big) .
  \end{multline}  
   To conclude, it remains to estimate the terms involving $\partial_{t} \phi$ in the right hand side of the above estimate.
     By using the elliptic equation \eqref{poissondt} and standard energy estimates,  we get that
     $$  \| \eps \nabla \partial_{t} \phi (t)  \|_{H^{m-1}_{\eps}}  + \|  \partial_{t} \phi  (t)\|_{H^{m-1}_{\eps}} \leq C(C_a, M)\big (  \| \partial_{t} n (t) \|_{H^{m-1}_{\eps}}
      +  \| \phi (t) \|_{H^{m-1}_{\eps}} + \eps^{K+ 1} \big)
     $$ 
      for every $t \in [0, T^\eps]$.
     Therefore, we obtain in particular that
     $$ \eps  \|  \partial_{t} \phi  (t)\|_{H^{m-1}_{\eps}} \leq C(C_a, M)\big (   \eps \| \partial_{t} n (t) \|_{H^{m-1}_{\eps}} +   \| \phi( t)  \|_{H^{m-1}_{\eps}} + \eps^{K+ 1} \big).$$
      We can then use the first line of \eqref{EPP} to express $\eps \partial_{t}n$ in terms of space derivatives, this yields  that on $ [0, T^\eps]$, we have
    \begin{equation}
    \label{fin2}   \eps  \|  \partial_{t} \phi (t)  \|_{H^{m-1}_{\eps}} \leq C(C_a, M)\big (    \| ( n, u) (t) \|_{H^{m}_{\eps}} +   \| \phi (t) \|_{H^{m-1}_{\eps}} + \eps^{K+ 1} \big).
    \end{equation}
       Finally, we can estimate $ \| \sqrt{\eta'}\, \eps \partial_{t}\phi\|_{L^2_{T}H^{m-1}_{\eps}}$.  This follows again from estimates on the elliptic equation
        \eqref{poissondt}. We just use the  weight $\eta'$ in the estimates as in the derivation of \eqref{3.37}.
         This yields for $t \in [0, T^\eps]$,  
       $$ \| \sqrt{\eta'}\,  \partial_{t}\phi (t) \|_{H^{m-1}_{\eps}} \leq C(C_{a}, M)  \big(\| \sqrt{\eta'} \partial_{t}n  (t)\|_{H^{m-1}_{\eps}} +  \|\sqrt{\eta'}\, \phi (t) \|_{H^{m-1}_{\eps}}
        + \eps^{K}\big).$$
         Consequently, by using again the first line of \eqref{EPP},  to express $\partial_{t}n$, we obtain
      \begin{equation}
      \label{fin3}  \eps   \| \sqrt{\eta'}\,  \partial_{t}\phi (t) \|_{H^{m-1}_{\eps}} \leq C(C_{a}, M)  \big( \| \sqrt{\eta'} (n, u, \phi)  (t)\|_{H^{m}_{\eps}} + \| (n,u, \phi ) (t)\|_{H^m _{\eps}}  + \eps^{K+1}\big).
      \end{equation}
         
 From \eqref{fin1}, \eqref{fin2}, \eqref{fin3}, we have   thus proven that for every $T \in [0, T^\eps]$, we have 
 \begin{multline*}
 \sqrt{\mu} \big(\|(n,u,\phi)\|_{L^\infty_{T} H^m_{\eps}}^2 + \| \sqrt{\eta'} (n,u, \phi) \|_{L^2_{T}H^m_{\eps}}^2 \big)
 \\ \leq C(C_{a}, M)  \Big(  \eps^{2K} +   \mu    \|\sqrt{\eta'} (n,u, \phi)\|_{L^2_{T}H^{m}_{\eps}}^2  
  +  \| (n, u, \phi) \|_{L^2_{T}H^m_{\eps}}^2  \Big) .
  \end{multline*}  
  For $\mu$ sufficiently small, the singular term $  \mu    \|\sqrt{\eta'} (n,u, \phi)\|_{L^2_{T}H^{m}_{\eps}}^2$ can be absorbed in the left hand side, 
   this yields
   $$  \sqrt{\mu} \|(n,u,\phi)\|_{L^\infty_{T} H^m_{\eps}}^2  \leq C(C_{a}, M)  \Big(  \eps^{2K}  +  \| (n, u, \phi) \|_{L^2_{T}H^m_{\eps}}^2  \Big)$$
    and hence  from the Gronwall inequality, we obtain
    $$  \|(n,u,\phi) (T)\|_{ H^m_{\eps}}^2 \leq C(C_{a}, M) \eps^{2 K} e^{ T  { C(C_{a}, M) \over \sqrt{\mu}} }, \quad \forall T \in [0, T^\eps ].$$ 
     We have thus proven \eqref{Hcmeps}.
     
       The parameter $\mu$ can now be considered as fixed. By standard continuation arguments, we next obtain that for $\eps$ sufficiently
        small $T^\eps \geq T_{0}$ and that  on $[0, T_0]$, we have
        $$  \sup_{T\in [0, T_0]}  \|(n,u,\phi) (T)\|_{ H^m_{\eps}}^2 \leq C(C_{a}, M) \eps^{2 K} e^{ T_0 { C(C_{a}, M) \over \sqrt{\mu}} }.$$
        This ends the proof of Theorem \ref{main}.

  \bibliographystyle{plain}
\bibliography{sheath}

\begin{thebibliography}{10}

\bibitem{AMB}
A.~Ambroso.
\newblock Stability for solutions of a stationary {E}uler-{P}oisson problem.
\newblock {\em Math. Models Methods Appl. Sci.}, 16(11):1817--1837, 2006.

\bibitem{AMR}
A.~Ambroso, F.~M{\'e}hats, and P.~A. Raviart.
\newblock On singular perturbation problems for the nonlinear {P}oisson
  equation.
\newblock {\em Asymptot. Anal.}, 25(1):39--91, 2001.

\bibitem{BenzoniSerre}
S.~Benzoni-Gavage and D.~Serre.
\newblock {\em Multidimensional hyperbolic partial differential equations}.
\newblock Oxford Mathematical Monographs. The Clarendon Press Oxford University
  Press, Oxford, 2007.
\newblock First-order systems and applications.

\bibitem{CG}
S.~Cordier and E.~Grenier.
\newblock Quasineutral limit of an {E}uler-{P}oisson system arising from plasma
  physics.
\newblock {\em Comm. Partial Differential Equations}, 25(5-6):1099--1113, 2000.

\bibitem{GVHKR}
D.~G\'erard-Varet, D.~Han-Kwan, and F.~Rousset.
\newblock {Quasineutral limit of the Euler-Poisson system for ions in a domain
  with boundaries}.
\newblock {\em Indiana Univ. Math. J.}, 62:359--402, 2013.

\bibitem{Goo}
J.~Goodman.
\newblock Nonlinear asymptotic stability of viscous shock profiles for
  conservation laws.
\newblock {\em Arch. Rational Mech. Anal.}, 95(4):325--344, 1986.

\bibitem{LL}
M.~Lieberman and A.~Lichtenberg.
\newblock {\em Principles of plasma discharges and materials processing}.
\newblock Cambridge Univ Press, 1994.

\bibitem{NOS}
S.~Nishibata, M.~Ohnawa, and M.~Suzuki.
\newblock Asymptotic stability of boundary layers to the {E}uler-{P}oisson
  equations arising in plasma physics.
\newblock {\em SIAM J. Math. Anal.}, 44(2):761--790, 2012.

\bibitem{Rie}
K-U Riemann.
\newblock {The Bohm criterion and sheath formation}.
\newblock {\em J. Phys. D: Applied Physics}, 24(4):493, 1991.

\bibitem{SS}
M.~Slemrod and N.~Sternberg.
\newblock Quasi-neutral limit for {E}uler-{P}oisson system.
\newblock {\em J. Nonlinear Sci.}, 11(3):193--209, 2001.

\bibitem{Suz}
M.~Suzuki.
\newblock {Asymptotic stability of stationnary solutions to the Euler-Poisson
  equations arising in plasma physics}.
\newblock {\em Kinet. and Relat. Mod.}, 4(2):569--588, 2011.

\end{thebibliography}

\end{document}